\documentclass[num-refs]{wiley-article}

\usepackage{siunitx}
\usepackage{amsmath,amssymb,amsfonts}
\usepackage{algorithmic}
\usepackage{amssymb}
\usepackage{mathtools}
\usepackage{tabularx}
\usepackage{comment}
\usepackage{appendix}

\newcommand{\bx}{\bm{x}}
\newcommand{\bu}{\bm{u}}
\newcommand{\bw}{\bm{w}}
\newcommand{\bz}{\bm{z}}
\newcommand{\bv}{\bm{v}}
\newcommand{\bp}{\bm{p}}
\newcommand{\be}{\bm{e}}
\newcommand{\dtmpc}{SDD-TMPC}
\newcommand\norm[1]{\left\lVert#1\right\rVert}
\newcommand{\Ts}{T^{\textrm{s}}}
\newcommand{\we}{w^{\text{e}}} 
\newcommand{\wpp}[1]{w^{\text{e,}#1}_t} 
\newcommand{\maxu}{^{\max}}

\newtheorem{my_def}{Definition}

\papertype{RESEARCH ARTICLE}

\title{State-Dependent Dynamic Tube MPC: A Novel Tube MPC Method with a Fuzzy Model of Disturbances}

\abbrevs{MPC, model predictive control.}

\author[1\authfn{1}]{Filip Surma MSc}
\author[1\authfn{1}]{Anahita Jamshidnejad PhD}

\contrib[\authfn{1}]{Equally contributing authors.}

\affil[1]{Aerospace Engineering Faculty, TU Delft, 2629 HS, Delft, South Holland, The Netherlands}

\corraddress{Filip Surma, Aerospace Engineering Faculty, TU Delft, 2629 HS, Delft, South Holland, The Netherlands}
\corremail{f.surma@tudelft.nl}

\fundinginfo{This research has been supported by the NWO Talent Program Veni project ``Autonomous drones flocking for search-and-rescue'' (18120), which has been financed by the Netherlands Organisation for Scientific Research (NWO).}

\runningauthor{Surma and Jamshidnejad}

\begin{document}

\begin{frontmatter}
\maketitle

\begin{abstract}
Most real-world systems are affected by external disturbances, which may 
be impossible or costly to measure. 
For instance, when autonomous robots move in dusty environments, the perception of their sensors is disturbed.  
{Moreover, uneven terrains can cause ground robots to deviate from their planned trajectories.} 
Thus, learning the external disturbances and incorporating this knowledge into the future predictions  
in decision-making can significantly contribute to improved performance. 
Our core idea is to learn the external disturbances that vary with the states of the system,  
and to incorporate this knowledge into a novel formulation 
for robust tube model predictive control (TMPC). 
Robust TMPC provides robustness to bounded disturbances 
considering the known (fixed) upper bound of the disturbances,   
but it does not consider the dynamics of the disturbances. This can lead to highly conservative solutions. 
We propose a new dynamic version of robust TMPC (with proven robust stability), called state-dependent dynamic TMPC (\dtmpc), 
which incorporates the dynamics of the disturbances  
into the decision-making of TMPC. 
In order to learn the dynamics of the disturbances as a function of the system states, 
a fuzzy model is proposed. 
We compare the performance of \dtmpc, MPC, and TMPC via simulations,   
in designed search-and-rescue scenarios. 
The results show that, while remaining robust to bounded external disturbances,  \dtmpc\ 
generates less conservative solutions and remains feasible in more cases, compared to TMPC. 

\end{abstract}
\end{frontmatter}
\section{Introduction}
\label{sec:intro}
Model predictive control (MPC) \cite{MPCbook,MPC_industry_survey} 
is a state-of-the-art control approach that can optimize 
multiple control objectives, handle state and input constraints, and provide guarantees on 
stability and feasibility. MPC heavily relies on a model that predicts the evolution of the states 
of the controlled system 
across a given prediction horizon. Thus, unmodeled external disturbances 
may lead to dangerous situations in real-life applications. 
For instance, search-and-rescue (SaR) robots should operate autonomously 
in unknown environments that are prone to external disturbances 
\cite{USAR_survey_2013,USAR_control_and_perception_2020}. 
Therefore, for mission planning SaR robots need control methods  
that next to optimize the objectives of the SaR mission (e.g., maximizing 
the coverage of the area in the smallest possible time) 
and satisfying the constraints, 
also, deal with external disturbances that pose risks to the mission.
Thus, MPC methods that can systematically handle the disturbances are promising for SaR robots.%

Various extensions to standard MPC, including robust and stochastic MPC, have been developed \cite{SMPCsurvey}. 
In this paper, we focus on robust tube MPC (TMPC) \cite{TMPC}, which provides robustness 
to bounded disturbances{,} without significantly increasing the online computation time. 
%
%
%
%
TMPC generates control inputs that are composed of two parts: a nominal and an ancillary control input. 
The nominal states and control inputs of the system are determined by solving 
the nominal version of the MPC optimization (obtained by excluding external disturbances). 
The nominal states  determine the centroids of a tube (i.e., a set of possible future states given chosen inputs and bounded disturbances)  
that propagates across the prediction horizon. 
The cross sections of the tube lie within the state space 
of the system, and as long as the realized states (i.e., the states in {the} presence of the external disturbances)
at the corresponding time step remain in these cross sections, robustness
 is guaranteed. 
The ancillary controller ensures that the realized states remain inside the tube.%

\begin{figure}
    \centering
    \includegraphics[width=1\textwidth]{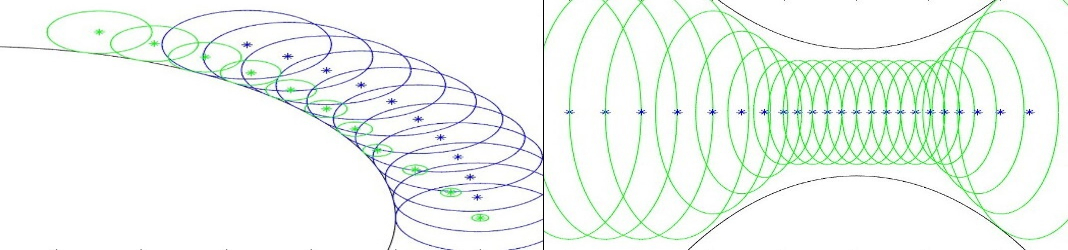}
    \caption{Comparison of \dtmpc\ and TMPC, where the sequences of the circular areas 
    in both plots illustrate the evolution of the tube based on a prediction at the current time step. 
    \textbf{Left plot:} The tube in \dtmpc\ (shown via the smaller varying green circles) is dynamic, 
    and always a subset of the tube of regular TMPC (shown via the larger blue circles). 
    Thus, \dtmpc\ allows the system to perform closer to the boundaries of hard constraints, 
    and to potentially improve the performance, without violating these constraints. 
    \textbf{Right plot:} The nominal trajectory (shown via the blue stars) and the tube of \dtmpc\ 
    (shown via the green varying circles) are illustrated when the states 
    are constrained to remain within a set shown by the black boundary curves. \dtmpc\ 
    allows, e.g., by reducing the speed of a robot, to formulate a feasible problem and provide solutions for it, 
    and to safely move through the narrow corridor.
}
    \label{fig:illustration}
\end{figure}

The nominal MPC in TMPC is not aware of the dynamics of the external disturbances. 
Thus, while the simplicity of the nominal optimization problem leads to a lower
computation time, this may be at the cost of compromising the performance 
and generating overly conservative solutions, especially for systems with 
nonlinear dynamics and prone to dynamic disturbances. 
Moreover, existing MPC methods, including TMPC, do not naturally allow for incorporation of 
the expert knowledge. This gap should be addressed, particularly when extensive expert knowledge is 
available, 
but humans cannot or should not participate in real-time control procedures,  
for instance for autonomous onboard control of SaR robots.%

%
%
%
Correspondingly, the main contributions of this paper are:
\begin{itemize}
    \item 
    We introduce \dtmpc, a novel extension to TMPC. 
    \dtmpc\ incorporates the dynamics of the external 
    disturbances into its nominal MPC formulation and generates tubes that are subsets of the tube of TMPC. 
    \dtmpc\ reduces the conservatives of the solutions  (see Figure~\ref{fig:illustration}) 
    and results in improved performance, with affordable  
    computational efforts. We also prove the robust stability of \dtmpc.
    \item 
    We propose a novel way to learn the dynamics of 
    disturbances, initiated by the existing expert knowledge,  
    and using a fuzzy inference system. While other modeling approaches, e.g., 
    neural networks, require an extensive data set to train the initial model, 
    a fuzzy inference system {(FIS)} can easily be initiated with human  
    knowledge and then be tuned online to adapt to the changing environment. 
    \item 
    We compare the performance of \dtmpc\ with regular MPC and TMPC for a linear system with state-dependent external disturbances that correspond to an autonomous robot navigating in an unknown environment. The results confirm that 
    \dtmpc\ outperforms both regular MPC and TMPC 
    by generating less conservative control inputs that still guarantee the satisfaction 
    of the hard constraints.
\end{itemize}
The rest of the paper is structured as follows. 
In Section~\ref{sec:background}, we discuss the related work 
and identify the open challenges and limitations of the state-of-the-art methods. 
In Section~\ref{sec:methodologies}, we explain our proposed approaches, 
including \dtmpc, 
and the fuzzy model of disturbances.
We also provide proof for the stability of \dtmpc. 
In Section~\ref{sec:case_study}, we implement and compare the performance of MPC, TMPC and \dtmpc\ 
via computer simulations for a ground 
robot. 
Finally, Section~\ref{sec:conclusion} concludes the paper and suggests topics for future research. 
The abbreviations and mathematical notations that are frequently used in this paper are listed in Table \ref{tab:abbreviations} and Table \ref{tab:math}, respectively.

\begin{table}[]\caption{Frequently used abbreviations and their explanation}
\label{tab:abbreviations}\begin{tabular}{|l|l|l|l|}
\hline
\multicolumn{1}{|l|}{Abbreviation} & \multicolumn{1}{l|}{Explanation} & \multicolumn{1}{l|}{Abbreviation} & \multicolumn{1}{l|}{Explanation}                  \\ \hline
A$i$                               & The $i^\text{th}$ assumption         &  SaR                               & Search and Rescue                           \\
FIS                                & Fuzzy inference system               &   SDD-TMPC                          & State-dependent dynamic tube model predictive control\\
GA                                 & Genetic algorithm                    &   TSK-FIS                           & Takagi-Sugeno-Kang fuzzy inference system          \\
MPC                                & Model predictive control             &  TMPC                              & Tube model predictive control                      \\
PS                                 & Particle swarm algorithm             & ZOH                              & Zero-order hold           \\
\hline
\end{tabular}
\end{table}

\begin{table}[]\caption{Frequently used mathematical notations and their definition}\resizebox{\textwidth}{!}{
\label{tab:math}
\begin{tabularx}{\textwidth}{|l|X|l|X|}
\hline
\multicolumn{1}{l|}{Notation} & \multicolumn{1}{l|}{Definition}                                      
& \multicolumn{1}{l|}{Notation} & \multicolumn{1}{l|}{Definition}                  \\ \hline

$k$                               & Time step counter in the discrete time domain  
& $\mathcal{K}$                      & A class of continuous and strictly increasing functions starting at zero \\

$N$                                & Prediction horizon & 
$\mathcal{K}_\infty$               & A subclass of $\mathcal{K}$\\

$\bz_k$                           & Nominal state vector 
& $\mathbb{W}(\bx_k)$               & Set of all possible disturbances given state $\bx_k$ of the system at time step $k$ 
\\

$\bx_k$                   & State vector at time step $k$ for a given dynamical system 
 & $\Bar{\mathbb{W}}_k$                & Set of all possible disturbances for all possible states estimated at time step $k$\\

$e_k$                              & Error between the real and the nominal state for time step $k$              
&  $\mathbb{X}$                      & Admissible set for the states  
\\

$\bw(\bx_k)$                      & Disturbance vector when the state of the system is $\bx_k$             
       &  $\mathbb{Z}_f$                    & Terminal set for the nominal states                          \\

$\bv_k$                           & Nominal input vector      
&   $\mathbb{U}$                       & Admissible set for the control inputs     \\

$\bu_k$                            & Input vector   
&  $\mathbb{N}$                       & Set of natural numbers \\

 $\tilde{\bz}_k$                    & Sequence of the nominal state vectors across the prediction horizon (i.e., $\{\bz_{k+1}, \ldots, \bz_{k+N}\}$) estimated at time step $k$                                                            
&  $\mathbb{R}$                     & Set of real numbers\\

 $\tilde{\bx}_k$                    & Sequence of the state vectors across the prediction horizon (i.e., $\{\bx_{k+1}, \ldots, \bx_{k+N}\}$) estimated at time step $k$ 
&
$\Bar{\mathbb{E}}_k$                 & Error set estimated for time step $k$, assuming an autonomous evolution of the error set (excluding the influence of the disturbances at previous time step $k-1$) \\
$\tilde{\bv}_k$                   & Sequence of the nominal input vectors across the prediction horizon (i.e., $\{ \bv_k,\ldots,\bv_{k+N-1}\}$) estimated at time step $k$ 
& $\mathbb{E}_k$                       & Set of all possible errors between the real and the nominal state for time step $k$    \\

$\tilde{\bu}_k$                    & Sequence of the input vectors across the prediction horizon (i.e., $\{ \bu_k,\ldots,\bu_{k+N-1}\}$) estimated at time step $k$ 
&    $\mathbb{T}\maxu_{k}$              & Tube of regular robust TMPC for prediction horizon $\{k+1, \ldots , k+N\}$ with a fixed cross section area $\mathbb{E}\maxu$  \\

$\pi(\cdot)$                      & Ancillary control law 
&  $\mathbb{T}_k$                     & Tube of \dtmpc\ for  prediction horizon $\{k+1, \ldots , k+N\}$ with a dynamic cross section area $\mathbb{E}_i$ for $i=k+1, \ldots,k+N$\\

 $\epsilon(\cdot)$                 & Autonomous error dynamic function, which determines $\Bar{\mathbb{E}}_k$ based on $\mathbb{E}_{k-1}$ 
& $\oplus$                          & Minkowski summation \\

$(\cdot)_{i|k}$                           & prediction of a dynamic variable for time step $i$, given the measurements at time step $k$   & $\ominus$                         & Minkowski difference                              \\


\hline   
\end{tabularx}}
\end{table}

\section{Related work}
\label{sec:background}

\paragraph{Modeling via fuzzy inference systems} 
FIS can approximate any nonlinear function, 
when the inputs of the function are bounded \cite{universalApproximationFIS}. 
Various supervised learning algorithms, e.g., gradient descent \cite{gradient_FIS_learning}, 
learning from example \cite{LFE}, and genetic algorithm (GA) \cite{FISlearningGA} 
 have been proposed for the FIS to learn such functions from data. 
 One of the most common FISs used in various applications, including robotics, 
 is Takagi-Sugeno-Kang (TSK) FIS. 
 TSK FISs can be updated online in a stable way using, e.g., 
 reinforcement learning 
 (see applications to robot navigation \cite{fuzzyRLnavigation}, balancing a bipedal robot \cite{FRLbipedalRobot}, 
 controlling a continuum robot \cite{FRLcontinuumRobot}). 
 To the best of our knowledge, FISs have never been used to model 
the dynamics of external disturbances, especially for providing robustness for model-based 
control approaches, e.g., MPC.%

\paragraph{Robust MPC}

Linear MPC is a well-known, theoretically mature control approach 
that  has successfully been implemented 
in many applications \cite{oldMPCsufrvey,newMPCsufrvey}. {Linear MPC was utilized to regulate multiple vehicles across various domains such as spacecraft altitude control \cite{LMPCaerospace}, ground vehicle steering \cite{LMPCsteering}, and drone trajectory tracking \cite{LMPCdroneTracking}. The LMPC framework incorporates a linear system, linear constraints, and a quadratic cost function.
In real-life problems, including SaR, in order to deal with 
nonlinear cost functions and nonlinear constraints, a nonlinear version of MPC may be used. 
Nonlinear model predictive control (NMPC) has been implemented in various applications including exploration of a grid world with multiple agents \cite{DMPCexlporation}, decision making for simultaneous localization and mapping (SLAM) \cite{MPCSLAM}, obstacle avoidance \cite{sliding-modeNMPC}, and global path planning \cite{pathAlteringMPC}. 
Although MPC is robust to small disturbances (due to operating upon state feedback),  
linear and nonlinear MPC cannot deal with large external disturbances \cite{SMPCsurvey}. 
Therefore, robust MPC approaches have been introduced.%


One of the most attractive robust MPC methods is robust tube MPC (TMPC) \cite{TMPC}, where a robust control 
strategy that is designed offline is used to keep the state trajectory within an invariant tube, 
the center-line of which is a nominal trajectory of the states that is determined online.  
For the sake of brevity of the notations, we refer to robust TMPC as TMPC in the rest of the paper. 
The nominal state trajectory can be determined for the system whenever the controller has a perfect model of the system and there are no disturbances. 
Although the nominal states are predictable in such cases, the actual states cannot be known in advance because the external disturbances are unknown. However, since the disturbances are assumed to be bounded and 
thus the maximum error between the actual and the nominal states can be determined, it is possible to compute 
a sequence $\mathbb{T}\maxu_{k}$  of $N$ sets per time step $k$, where each set $\{z_{i|k}\}\oplus\mathbb{E}\maxu$ 
 within this sequence includes all the possible actual states for a given time step $i$ within the prediction window 
 $\{ k+1, \ldots, k+N\}$, and $N$ is the prediction horizon of TMPC. 
Then the state constraints are guaranteed to be satisfied within the prediction window, if all the states within 
sequence $\mathbb{T}\maxu_{k}$ satisfy the constraints. 
In order to prevent the actual states from deviating significantly from the corresponding nominal states, an ancillary control law, 
e.g., an error state feedback controller \cite{TMPC}, another MPC controller \cite{TNMPBwithMPCasAnciallary}, or a sliding mode control \cite{sliding-modeNMPC}, is used to reduce the error between the actual and the nominal state. 

There are two common ways of designing and implementing TMPC \cite{TVTMPC}: 
\begin{itemize}
    \item  
    The first approach involves determining a single error set (i.e., tube) for the entire prediction window based on the error dynamics and the selected ancillary control law, 
    where this error set serves as a positive robust invariant set. In other words, if the current error belongs to this positive robust invariant set, then all the future errors will remain within the set. This approach benefits from a low online computational complexity, since the positive robust invariant set can be computed offline. Moreover, tighter constraints can be imposed on the nominal set, such that if the center of the positive robust invariant set is inside the tightened constraints, then the actual state remains in the original feasible set. However, using this approach may result in more restrictive control inputs. 
    \item The second approach involves computing the error set per time step within the prediction window. This approach results in a sequence of regions, known as reachable sets, which are the smallest sets of possible states for a system prone to external disturbances that guarantee that the states remain in these sets for all time steps. This approach may yield less conservative solutions for TMPC, but requires additional online computations compared to the first approach. 
\end{itemize}}

Some variations of TMPC for nonlinear systems have been proposed. 
Nonlinear TMPC is much more challenging than linear TMPC, particularly because it is not trivial to choose an ancillary control law and to design a tube. 
One way to reduce the computational complexity of nonlinear TMPC is by using parameterized TMPC, such as the approach in  \cite{FPTMPC}, 
which reduces the required online optimization into a linear programming problem.    
Another alternative is to employ ellipsoids as tubes, instead of polytopic sets (see, e.g., \cite{DTMPCelliseTimeVarying}).  
Some examples of nonlinear TMPC can be found in \cite{TNMPBwithMPCasAnciallary,sliding-modeNMPC,deepLearningTubes,stateDepdentNoiseMPC_GP,SDSMPC_implmentedAtmosphericPressurePlasmaJetsForPlasmaMedicine}. 

The first implementation of robust MPC (although not TMPC)
when state-dependent disturbances exist includes \cite{first_RMPC_SD}. 
In \cite{STRMPC_as_cone_program} the optimization problem of robust MPC 
has been handled as a 2nd-order cone program. This approach, however, 
can only be used for linear systems subject to a certain group of 
additive disturbances (i.e., disturbances defined as the summation 
of an independent component from a polytope and a state and
input-dependent component bounded via a non-convex inequality). 
Similarly, in \cite{RMPC_SD_nondeadcreasing_w} robust MPC has been discussed for a 
special class of state-dependent disturbances. 
However, our proposed TMPC-based approach, i.e., \dtmpc, is not limited to any type of nonlinearity.
Authors in \cite{sliding-modeNMPC} propose a stable controller for  an 
agent that should avoid obstacles in an environment with disturbances that are proportional 
to the square of its velocity. The proposed method, however, is restricted 
to a sliding-mode ancillary control law, and thus can only be used for 
feedback-linearizable or minimum-phase systems. 
In \cite{ILFRMPC}, a dynamic tube is used and 
is parameterized as a sublevel set of incremental Lyapunov functions. 
This method is extended in \cite{SDSMPC_with_ILF}, 
where chance constraints and state or input-dependent disturbances are included. 
This paper has the most similarities with \dtmpc, since it 
uses a dynamic tube that evolves in time, and tightens the 
constraints online. 
Using a sub-level set of incremental Lyapunov functions 
reduces the conservativeness with a small increase in parameters and equations (and probably in computations) compared to other approaches in the literature. 
The parameterization, however, restricts the shape of the tube, 
which potentially leads to conservative decisions. 
In \cite{deepLearningTubes}, a neural network has been used to learn 
the dynamics of the tube of robust TMPC. 
The neural network, however, cannot be initialized without an extensive dataset, 
and thus another controller should be used until a sufficient number of data is gathered. 
Moreover, no proof has been provided that the trajectory of the states 
remains inside the tube. 
In \cite{stateDepdentNoiseMPC_GP}, a Gaussian process is used to learn 
the disturbance set. 
The stability of the algorithm for linear systems has been proved. However, this method requires heavy offline computations, and discretization 
of the state space, and may thus become intractable for systems with large state spaces. 

\dtmpc\  is proved to be stable under standard assumptions. 
It can be used in unknown and time-varying environments 
and can be initialized with intuitive rules and be tuned online. 
Unlike most state-of-the-art methods, \dtmpc\ 
is not limited to any specific class of nonlinearity and 
does not restrict the choice of the ancillary control law. 
Moreover, unlike state-of-the-art methods that assume the  
bounds or the probability distribution of the disturbances 
are known, for \dtmpc\ it is possible to learn the disturbances online.%

\setlength{\belowdisplayskip}{5pt} \setlength{\belowdisplayshortskip}{3pt}
\setlength{\abovedisplayskip}{5pt} \setlength{\abovedisplayshortskip}{3pt}
\section{\dtmpc: Idea, formulation, and stability}
\label{sec:methodologies}

In this section, we describe and formulate the problem and  
discuss the proposed methods for \dtmpc.%

\subsection{Problem statement}
\label{sec:problem_statement}

We consider a dynamic system that is described in discrete time with the 
state vector $\bx_k$ and control input $\bu_k$, and is affected by additive state-dependent external disturbances 
$\bw(\bx_k)$ at time step $k \in \mathbb{N}$, where $\bx_k \in \mathbb{X} \ \subseteq \mathbb{R}^n$, 
$\bu_k \in \mathbb{U} \ \subseteq \mathbb{R}^m$, and $\bw(\bx_k) \in \mathbb{W}(\bx_k) \ \subseteq \mathbb{R}^n$. 
While the admissible sets $\mathbb{X}$ and $\mathbb{U}$ for, respectively, the state and the 
control input are static, we allow the admissible set $\mathbb{W}(\bx_k)$ of the 
external disturbances to vary as a function of the state. 
Learning (an approximate of) the admissible set $\mathbb{W}(\bx_k)$, such that it bounds the 
external disturbances per state, instead of considering a set that bounds all potential 
external disturbances for the entire admissible set $\mathbb{X}$ in time, 
will introduce dynamics and reduced conservativeness into \dtmpc. 

The dynamic system is given by:
\begin{equation}
\label{eq:system}
    \bx_{k+1}=f(\bx_k,\bu_k)+\bw(\bx_k)
\end{equation} {where $f(\cdot)$ is a time-invariant time-discrete nonlinear Lipschitz function.}

We assume that the states are measured by perfect sensors, i.e., per time step $k$ 
the value of the state $\bx_k$ is perfectly known, whereas the external disturbances 
are unknown and may take any arbitrary value within 
$\mathbb{W}(\bx_k)$.
The aim is to control the dynamics of system \eqref{eq:system}, such that a given cost function  $J(\cdot)$ is minimized 
across a prediction horizon of size $N$, 
while it is guaranteed for the controlled system that for all admissible  external disturbances the hard constraints are always satisfied. 
Thus, for time step $k$ we have: 
\begin{align}
    \label{eq:original_MPC} 
    J^*(\bx_k)=&\min_{\tilde{\bx}_k, \tilde{\bu}_k} J(\tilde{\bx}_k,\tilde{\bu}_k) \\
    &\textrm{s.t. for $i = k, \ldots, k+N-1$} \nonumber\\
    &\quad  \bx_{i+1|k}=f(\bx_{i|k},\bu_{i|k})+\bw(\bx_{i|k}) \nonumber\\
    &\quad \bx_{k|k} = \bx_k \nonumber\\
    & \quad  \bx_{i+1|k} \in \mathbb{X} \nonumber\\ 
    & \quad  \bu_{i|k} \in \mathbb{U} \nonumber\\ 
    & \quad  \bw(\bx_{i|k}) \in \mathbb{W}(\bx_{i|k}) \nonumber
\end{align}
with $\tilde{\bx}_k = [\bx^\top_{k+1|k},\ldots,\bx^\top_{ k+ N|k}]^\top$, 
$\bx_{i|k}$ for $i>k$ the  value of $\bx_i$ predicted at time step $k$,  
$\tilde{\bu}_k = [\bu^\top_{k|k},\ldots,\bu^\top_{ k+ N - 1 |k}]^\top$,  
$\bu_{i|k}$ for $i\geq k$ the  value of $\bu_i$ determined at time step $k$, 
and $J^*(\bx_k)$ an optimal value for the cost function obtained by solving the constrained minimization problem.
The cost function $J(\cdot)$ 
is composed of a stage cost that is accumulated across the 
prediction horizon $\{k, \ldots,k+N-1\}$, and a terminal cost that is computed at the terminal 
time step $k+N$. The cost function is continuous and finite for $\bx_i\in\mathbb{X},\ \forall i\in\{k, \ldots, k+N\}$.

\subsection{State-dependent dynamic tube MPC (\dtmpc)}
\label{sec:sddtmpc}

TMPC works based on the assumption of bounded external disturbances, where the boundary set $\mathbb{W}\maxu$  for the disturbances is known and given \cite{TMPC}. 
A main challenge for \dtmpc\ in solving \eqref{eq:original_MPC} per time step $k$ is that the dynamic 
set $\mathbb{W}(\bx_{i})$ (for $i = k, \ldots, k+N-1$) is not known (neither for the current time step nor 
for the future time steps) and should thus be estimated.  
The link between  the known static set $\mathbb{W}\maxu$ and the dynamic set $\mathbb{W}(\bx_i)$ of disturbances that 
should be estimated by \dtmpc\ for every prediction horizon $\{k + 1, \ldots, k+N\}$ is given by: 
\begin{equation}
\label{eq:link_disturbance_sets}
 \bigcup_{
\substack{
i=k,\ldots,k+N - 1
 }}
\mathbb{W}(\bx_i) \subseteq \mathbb{W}\maxu  
\end{equation}
We assume that $\mathbb{W}(\bx_i)$ always contains the origin.
Regular TMPC uses $\mathbb{W}\maxu$ to determine a sequence of $N$ control inputs 
and a tube $\mathbb{T}_{k}\maxu$ for the entire prediction window per time step $k$, for which the  
cross section $\mathbb{E}\maxu$ remains unchanged.  
The tube  $\mathbb{T}_{k}\maxu$ is a robust positive invariant set. 
The main difference between TMPC and our proposed approach, \dtmpc, is in the formulation of their 
nominal MPC. This allows \dtmpc\ to  make use of 
the dynamics of the external disturbances, in order to generate per time step $k$ a tube $\mathbb{T}_{k}$ with a generally 
time-varying cross section $\mathbb{E}_i$ for $i=k+1, \ldots,k+N$  
across the prediction horizon that  improves the control performance by generating less conservative solutions. 
Note that set $\mathbb{T}_{k}$ is a robust control invariant set, i.e., there is at least one control 
policy that prevents the state trajectory from leaving the tube. 
The nominal problem of \dtmpc\ is given by:

\begin{subequations}
\label{eq:optimization}
\begin{equation}
\label{eq:cost}
     V^*(\bx_k,\bz_k) = \min_{\Tilde{\bz}_k, \Tilde{\bv}_k , \mathbb{T}_k}\sum_{i=k}^{k+N-1}l(\bz_{i|k},\bv_{i|k})+V_\text{f}(\bz_{k+N|k})
\end{equation}
\begin{equation*}
     \text{ s.t. for $i = k+1, \ldots, k+N$: }
\end{equation*}
\begin{equation}
\label{eq:nominal_system}
    \bz_{i|k}=f(\bz_{i-1|k},\bv_{i-1|k})  
\end{equation}
\begin{equation}
\label{eq:2nd_tube}
    \Bar{\mathbb{E}}_{i|k} =\epsilon(\mathbb{E}_{i-1|k})
\end{equation}
\begin{equation}
\label{eq:FIS_tube}
    \Bar{\mathbb{W}}_{i-1|k} = \text{FIS} \left(\{\bz_{i-1|k}\}\oplus\mathbb{E}_{i-1|k} \right)   
\end{equation}
\begin{equation}
\label{eq:final_tube}
    \mathbb{E}_{{i|k}} = \Bar{\mathbb{E}}_{i|k}\oplus\Bar{\mathbb{W}}_{i-1|k}
\end{equation}
\begin{equation}
\label{eq:first_tube}
   \mathbb{E}_k=\{\bx_k-\bz_k\}
\end{equation}
\begin{equation}
\label{eq:feasible}
     \{\bz_{i|k}\}\oplus\mathbb{E}_{i|k} \subseteq \mathbb{X} 
\end{equation}
\begin{equation}
\label{eq:terminal}
  \bz_{N|k} \in \mathbb{Z}_\text{f}
\end{equation}
 \begin{equation}{
  \label{eq:tubeTk_l}
     \mathbb{T}_k = \left\{\{z_{k+1|k}\}\oplus\mathbb{E}_{k+1|k}, \ldots, \{z_{N|k}\}\oplus\mathbb{E}_{k+N|k} \right\}   } 
 \end{equation}
\begin{equation}
\label{eq:policy_feasible}
    \pi\left(\{\bz_{i|k}\}\oplus\mathbb{E}_{i|k},\bv_{i|k},\bz_{i|k}\right) \subseteq \mathbb{U}
\end{equation}
\end{subequations}
In \eqref{eq:cost}, $l(\cdot)$ is the  stage cost, $V_\text{f}(\cdot)$ is the terminal cost,  and 
$\Tilde{\bz}_k$ and $\tilde{\bv}_k$ are the trajectory/sequence of, respectively, the nominal states and the nominal control inputs within the prediction window $\{k+1, \ldots,k+N\}$. The realized states of the system follow the nominal state sequence  
when there are no external disturbances. 
The nominal states $\bz_{i|k}$ are predicted at time step $k$ for time step $i = k+1, \ldots, k+N$ according to \eqref{eq:nominal_system}. 
Moreover, $\oplus$ represents the Minkowski addition \cite{Calculting_tube}, which for every two given sets $A$ and $B$ 
is defined by: 
\begin{equation}
\label{eq:Minkowski}
    \mathbb{A}\oplus\mathbb{B} = \{a+b : a\in\mathbb{A},b\in\mathbb{B}\}
\end{equation} 

The error set $\mathbb{E}_{i|k}$ (for $i = k+1, \ldots, k+N$) for \dtmpc\ is computed in three steps:  
First, in \eqref{eq:2nd_tube} the system uses the autonomous error dynamic function $\epsilon(\cdot)$ to 
estimate the evolved error set, without considering the influence of the state-dependent disturbances  
on the errors. The autonomous error dynamic function is determined via the state-of-the-art methods and   
based on the nominal model $f(\cdot)$ of the system dynamics and the policy $\pi(\cdot)$,  
assuming no external disturbances exists. Note that the error vector is defined in the prediction window by 
$\be_{i|k} = \bx_{i|k} - \bz_{i|k}$ 
for $i = k+1, \ldots, k+N$. 
For details on how $\epsilon(\cdot)$ 
can be determined, see, e.g., \cite{TMPC}, for linear TMPC and \cite{sliding-modeNMPC} for nonlinear TMPC.   
%
%
Second, in \eqref{eq:FIS_tube} a mapping (in this case a fuzzy inference system, FIS) is used to determine an admissible set $\Bar{\mathbb{W}}_{i-1|k}$ per time step $i$ ($i= k+1, \ldots, k+N$) that contains all the possible disturbances at time step $i-1$ that can result 
in all the possible states at time step $i$, which are obtained based on the estimated nominal states (see \eqref{eq:nominal_system}) and the error set corresponding to the previous time step. 
Note that the mapping $\text{FIS}(\cdot)$ is designed, such that $\Bar{\mathbb{W}}_{i-1|k}$ is a subset of $\mathbb{W}\maxu$ for $i = k+1, \ldots, k+N$ (see Section~\ref{sec:Fuzzy_System} for the details). 
Third, the sets obtained via the previous two steps (i.e., via \eqref{eq:2nd_tube} and \eqref{eq:FIS_tube}) are combined in \eqref{eq:final_tube}. In other words, the set  $\mathbb{E}_{i|k}$ of all possible errors for time step $i$ is obtained by 
combining the autonomously evolved error set and the set of all possible disturbances from the previous time step that can result in further errors. 
Based on Definition~\ref{def:D2} given next, $\mathbb{E}_{i|k}$ for $i = k+1, \ldots, k+N$ is an estimation of a robust forward reachable set for $\mathbb{E}_{i-1|k}$.
\begin{my_def}
\label{def:D2}
Consider the autonomous system that is formulated via \eqref{eq:2nd_tube}-\eqref{eq:final_tube}.  
A one-step robust reachable set from the set of errors $\mathbb{E}_{i-1|k}$ that is estimated for time step $i-1$ (with $i = k+1, \ldots, k+N$) is defined based on (\cite{TubeLipsitz}) via:
\begin{equation*}
    \mathcal{R}_\epsilon(\mathbb{E}_{i-1 |k },\Bar{\mathbb{W}}_{i-1|k})= 
    \Big\{\be\in \left(\mathbb{X} \ominus \{\bz_{i|k}\} \right)   
    |\exists \be_{i-1 | k}\in \mathbb{E}_{i-1|k},\exists \bw(\bx_{i-1|k})\in\Bar{\mathbb{W}}_{i-1|k}:\be=\epsilon(\be_{i-1|k})+\bw(\bx_{i-1|k}) \Big\}
\end{equation*}
\end{my_def} 
The three steps explained above for estimation of the error set $\mathbb{E}_{i|k}$ (for $i = k+1, \ldots, k+N$) 
are illustrated in Figure~\ref{fig:3ce}. 
Constraint 
 \eqref{eq:first_tube} initializes 
the error set per time step when the nominal \dtmpc\ is solved. 
Constraint \eqref{eq:feasible} keeps the realized states, 
estimated based on the nominal state and the error values, 
inside the admissible state set $\mathbb{X}$. 
Finally, \eqref{eq:terminal} is a terminal constraint with $\mathbb{Z}_{\text{f}} \subseteq \mathbb{X}$, 
where $\mathbb{Z}_{\text{f}}$ is a control invariant set for the nominal system  
that guarantees the recursive feasibility.%

\begin{figure}
    \centering
    \includegraphics[width=1\textwidth]{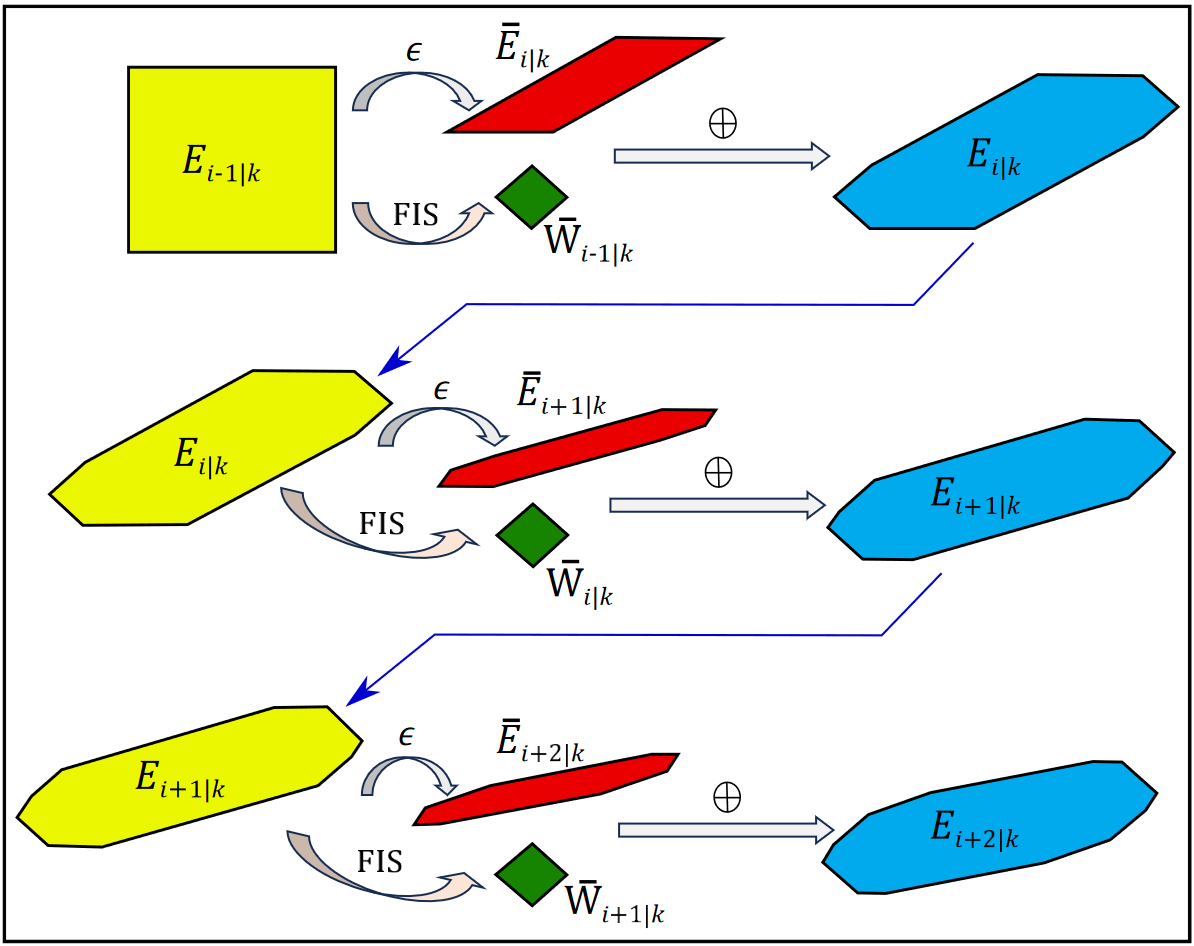}
    \caption{Illustration of the development of the error set $\mathbb{E}_{\cdot|k}$ for \dtmpc\ via \eqref{eq:2nd_tube}-\eqref{eq:final_tube} in $2$ dimensions: Each row corresponds to one prediction time step (illustrated in this figure for time steps $i$, $i+1$, $i+2$). 
    The $2$-dimensional sets in the first column (shown in yellow)  illustrate the error set $\mathbb{E}_{\cdot - 1|k}$ that has been determined at the previous time step. By propagating this error set through the autonomous error dynamics (i.e., mapping $\epsilon$) via \eqref{eq:2nd_tube}, the $2$-dimensional sets $\Bar{\mathbb{E}}_{\cdot|k}$  at the top side of the second column of each row (shown in red)  are generated. Moreover, the $2$-dimensional sets $\Bar{\mathbb{W}}_{\cdot - 1|k}$  at the bottom part of the second column in each row (shown in green) are obtained via the fuzzy inference system (mapping $\text{FIS}$) using \eqref{eq:FIS_tube}.  Finally, the red and green $2$-dimensional sets of the second column per row are combined using Minkowski summation to obtain the 
    error set $\mathbb{E}_{\cdot|k}$ via \eqref{eq:final_tube}. 
    This set is then used as the starting set for the next time step (i.e., at the start of the next row).}
    \label{fig:3ce}
\end{figure}

The dynamic tube $\mathbb{T}_k$ of \dtmpc\ for time step $k$ is given by \eqref{eq:tubeTk_l}.  
The optimization variables of \eqref{eq:optimization} are the nominal state sequence $\Tilde{\bz}_k$, 
the nominal control input sequence $\Tilde{\bv}_k$, 
and the \dtmpc\ tube $\mathbb{T}_k$ (i.e., an ordered set including the influence of the error sets across the entire prediction horizon). 
By including the tube as an optimization variable for the nominal \dtmpc\ problem, the optimizer determines 
solutions that foresee the dynamics of the error in the generation of the online  
nominal state trajectory (i.e., the center-line of tube $\mathbb{T}_k$).  
Therefore, the nominal state trajectory of \dtmpc\ is in general different from that  of regular TMPC (see Figure~\ref{fig:illustration}).   
Moreover, the \dtmpc\  policy $\pi: \mathbb{X}\times \mathbb{U} \times \mathbb{X} \rightarrow \mathbb{U}$, which receives the nominal state trajectory of \dtmpc\ as input, 
is incorporated within the optimization loop of \dtmpc\ via constraint \eqref{eq:policy_feasible}, as well as in the estimation of the autonomous error dynamic function as explained before. 
The policy $\pi(\cdot)$ of \dtmpc\ can be generated via state-of-the-art methods (see, e.g., 
\cite{TNMPBwithMPCasAnciallary,sliding-modeNMPC} for determining a policy that stabilizes the closed-loop system). 
Constraint \eqref{eq:policy_feasible} enforces the generated control inputs to remain inside the admissible control input set $\mathbb{U}$ 
for all possible realizations of the state. Thus the policy also plays a role in the computation of  
the online nominal state trajectory that is determined via \dtmpc. 
Note that for computing the policy in \eqref{eq:policy_feasible}, where the input of the function is a set, Definition~\ref{def:D1} given 
next is used.%
\begin{my_def}
Whenever a function $g(\cdot)$ takes a set $\mathbb{S}$ 
as input, then the output $\mathbb{O}$ is also a set defined by 
$\mathbb{O}\coloneqq \big\{g(i)  \big| i\in\mathbb{S}\big\}$. 
Thus, if $\mathbb{S}_1\subseteq\mathbb{S}_2$, 
then $g(\mathbb{S}_1)\subseteq g(\mathbb{S}_2)$. 
\label{def:D1}
\end{my_def} 
State-of-the-art TMPC methods, however, solve the nominal MPC problem without incorporating 
the policy $\pi(\cdot)$. In fact, only after determining the nominal state trajectory via an optimization procedure, 
the policy of TMPC is used outside of the optimization loop to generate the actual control input $\bu_k$.

In summary, the inclusion of the error evolution model into the optimization procedure of \dtmpc, 
which is done via incorporating the dynamic tube of \dtmpc\ in the optimization variables, as well as including the policy 
of \dtmpc\ into the constraints, are expected to result in solutions for \dtmpc\ that are less conservative than the solutions of TMPC.%


The optimization problem \eqref{eq:optimization} is in general nonlinear and non-convex. 
It is common to assume convexity for the admissible sets $\mathbb{X}$ and $\mathbb{U}$ for the states and 
the control inputs \cite{TMPC,stateDepdentNoiseMPC_GP}, but this is not the case in some applications (e.g., in collision avoidance \cite{sliding-modeNMPC,TubeLipsitz}, which is relevant for SaR).%

\subsection{Takagi-Sugeno-Kang fuzzy inference system (TSK-FIS) for modeling the dynamics of the external  disturbances}
\label{sec:Fuzzy_System}

A fuzzy inference system (FIS) can describe the dynamics of a (generally nonlinear) 
system via rules that are formulated 
via linguistic terms (e.g., very high), where this nonlinear mapping is interpretable --unlike, e.g., neural networks \cite{wide_fuzzy_book_2017}. 
The rules of a FIS can be generated based on experimental data or can be deduced directly from expert knowledge available in human language. In unknown environments where safety 
is crucial, a FIS with conservative rules may initially be used, and 
the rules can be tuned online to improve the performance.%

 The goal is to use a TSK-FIS for modeling the set of the external disturbances per time step $i-1$, where $i$ 
 belongs to the prediction horizon $\{k+1, \ldots,k+N\}$, as a function of the 
system states, i.e., to approximate $\mathbb{W}(\bx_{i-1|k})$ given in \eqref{eq:link_disturbance_sets} by $\bar{\mathbb{W}}_{i-1|k}$ using \eqref{eq:FIS_tube}. 
For the type of the set, we consider an ellipsoid or a polytope and assume that the origin always belongs to the set. 
Ellipsoids and polytopes are the most commonly used sets in related literature \cite{TMPC,DTMPCelliseTimeVarying} and --when being parameterized-- are bounded, if these parameters are bounded. 
Next, we parameterize the set and show the parameterized set by $\Bar{\mathbb{W}}_{i-1|k}(\bm{\Theta}_{i-1|k})$ where $\bm{\Theta}_{i-1|k}$ is the vector of all parameters at time step $i-1$. 
Therefore, the problem of approximating the disturbance set is translated into the problem of determining the vector $\bm{\Theta}_{i-1|k}$. 
In fact, this vector is returned by the FIS, 
which receives as input the corresponding state of the system (see \eqref{eq:FIS_tube}). 
In order to ensure robustness for \dtmpc\ to all values of the external disturbances, 
i.e., to make sure that \eqref{eq:link_disturbance_sets} holds, 
parameter vector $\bm{\Theta}_{i-1|k}$ should be determined such that the actual disturbance set 
$\mathbb{W}(\bx_{i-1|k})$ for all time steps $i\in\{k+1, \ldots,k+N\}$ is a subset of the approximate disturbance set 
$\Bar{\mathbb{W}}_{i-1|k}(\bm{\Theta}_{ i-1|k })$, which itself should be a subset of $\Bar{\mathbb{W}}\maxu$.

The rule base of the TSK-FIS of \dtmpc\ is composed of rules with the formulation given by \eqref{eq:final_rule}, 
where $\widetilde{X}^l$ is a fuzzy set that mathematically represents a linguistic term that describes the input variable $x$:
\begin{equation}
\label{eq:final_rule}
    \text{Rule } l:\text{ IF }x \text{ is } \widetilde{X}^l\text{, THEN } y^l \text{ is } h^l(x)
\end{equation}
The output generated by each rule concerns only one element of parameter vector $\bm{\Theta}_{i-1|k}$. 
More specifically, for $i\in\{k+1, \ldots,k+N\}$, the input $x$ is replaced by $\bx_{i-1|k}$, 
and the output $y^l$ is replaced by an element ${\theta}^l_{i-1|k}$ of the parameter vector. 
Note that more than one rule in the rule base may generate a candidate value for this element of 
the parameter vector. Therefore, the superscript $l$ is used to show that the value 
is generated for this element via the $l^{\text{th}}$ rule within the set of all $L$ rules that generate 
a candidate value for this element. 
Moreover, $h^l(\cdot)$ shows a generally nonlinear mapping from the input space to the output space 
(i.e., from the admissible set of the state variables of the dynamical system to the admissible set for element $\theta_{i-1|k}$ of the vector $\bm{\Theta}_{i-1|k}$).   
Then the final value for the parameter is computed by:
\begin{equation}
\label{eq:FIS}
    \theta_{i-1|k}=\frac{\Sigma^L_{l=1}\mu^l(\bx_{i-1|k})h^l(\bx_{i-1|k})}{\Sigma^L_{l=1}\mu^l(\bx_{i-1|k})}
\end{equation}
where $\mu^l(\cdot)$ is the membership function of fuzzy set $ \widetilde{X}^l$ and 
$L$ is the number of the rules in the rule base that generate a value for element $\theta_{i-1|k}$.

The robust control problem formulation is based on the assumption of existence of $\mathbb{W}\maxu$, which bounds all the possible disturbances. 
Thus, the evolution of the disturbance set generated via the TSK-FIS explained above is stable if 
the approximate disturbance set $\Bar{\mathbb{W}}_{i-1|k}(\bm{\Theta}_{i-1|k})$ is bounded by 
$\mathbb{W}\maxu$ for all $i\in\{k+1, \ldots,k+N\}$. 
The approximate set $\Bar{\mathbb{W}}_{i-1|k}(\bm{\theta}_{i-1|k})$ is bounded, whenever all 
elements of vector $\bm{\Theta}_{i-1|k}$ are  bounded (based on the assumption of an ellipsoid or polytope set or any other set that 
satisfies the condition of bounded output for bounded input). 
The membership functions in  \eqref{eq:FIS} are restricted to the interval $[0, 1]$ by definition, 
implying that to keep the elements of $\bm{\Theta}_{i-1|k}$ bounded, the functions $h^l(\cdot)$ (with $l\in\{1,\cdots,L\}$ ) 
should be formulated such that for all $\bx_{i-1|k}\in\mathbb{X}$ the function remains bounded. 
Finally, to ensure that the approximate set $\Bar{\mathbb{W}}_{i-1|k}(\bm{\Theta}_{i-1|k})$ that is designed to be bounded, 
is also a subset of $\mathbb{W}\maxu$, this set is defined as the union of the output of the TSK-FIS and $\mathbb{W}\maxu$.%



\subsection{Stability of \dtmpc}
\label{sec:stability_analysis}

In this section, we discuss the stability of \dtmpc, where our discussions are based on 
the following assumptions:  

\begin{enumerate}[label=A\arabic*]
    \item \label{as:tube}
    There exists 
    $\mathbb{T}\maxu_{k} = {\{ \{z_{k+1|k}\}\oplus\mathbb{E}\maxu, \ldots,  \{z_{N|k}\}\oplus\mathbb{E}\maxu\}}$ 
    that propagates across the prediction horizon at time step $k$ and is positive robust invariant under the policy $\pi(\cdot): \mathbb{X}\times \mathbb{U} \times \mathbb{X} \rightarrow \mathbb{U}$, i.e., when 
    the system follows policy $\pi(\cdot)$, the state remains inside this set.
    \item \label{as:invariant}
    There exists a terminal control invariant set $\mathbb{Z}_\text{f} \subseteq \mathbb{X}$ for the nominal system, such that  $\mathbb{Z}_\text{f}\oplus\mathbb{E}\maxu\subseteq\mathbb{X}$.
    \item \label{as:kappa}
    There exists a control law $\kappa: \mathbb{Z}_\text{f} \rightarrow \mathbb{U}$ for the nominal system, such that for $\bz_k \in \mathbb{Z}_\text{f}$ and $i=k,\ldots,k+N-1$:
    \begin{itemize}
        \item $f(\bz_{i|k},\kappa(\bz_{i|k}))\in\mathbb{Z}_\text{f}$ 
        \item $V_\text{f}(f(\bz_{i|k},\kappa(\bz_{i|k})))-V_\text{f}(\bz_{i|k})\leq-l(\bz_{i|k},\kappa( \bz_{i|k})) \qquad
        $ (where $V_\text{f} (\cdot)$ is the terminal cost as is given in \eqref{eq:cost})
        \item $\pi(\bx_{i|k},\kappa(\bz_{i|k}),\bz_{i|k})\in\mathbb{U}$,  
        $\bx_{i|k}\in \{\bz_{i|k}\}\oplus\mathbb{E}\maxu$
    \end{itemize}
    \item \label{as:k}
    There exists {$\mathcal{K}_\infty$} functions $\alpha_1(\cdot)$ and $\alpha_\text{{f}}(\cdot)$ that satisfy the following inequalities 
    (note that a function belongs to class {$\mathcal{K}$}, if it is continuous, zero at zero, and strictly increasing;  
    and a function belongs to class {$\mathcal{K}_\infty$}, if it is in class $K$ and is unbounded \cite{MPCbook}) ):
    \begin{itemize}
        \item $l(\bz_{i|k}, \bv_{i|k})\geq\alpha_1(|\bz_{i|k}|) \quad\forall \bz_{i|k}\in \mathbb{X},\forall \bv_{i|k}\in \mathbb{U} \qquad$ (where $l (\cdot)$ is the stage cost as is given in \eqref{eq:cost})
        \item $V_\text{f}(\bz_{i|k})\leq\alpha_\text{{f}}(|\bz_{i|k}|) \quad\forall \bz_{i|k}\in\mathbb{Z}_\text{f}$
    \end{itemize}\label{assumptions}
\end{enumerate}
Assumption~\ref{as:tube} implies that a stabilizable TMPC law can be determined for the system (i.e., the error does not go to infinity). 
Moreover, $\mathbb{T}\maxu_{k}$ is the tube with a constant cross section that is used in TMPC.
Assumption~\ref{as:invariant} (existence of a nominal invariant set) is a standard assumption in TMPC literature \cite{TMPC}. 
The first two items of Assumption~\ref{as:kappa} and Assumption~\ref{as:k}  are standard assumptions in MPC literature \cite{MPCbook} 
that are required to use the optimal cost as a Lyapunov function. 
The third item of Assumption~\ref{as:kappa} indicates that there exists a control law in the terminal set, 
such that the input constraints are satisfied.

We first give some theorems that are used in the discussions on stability.
\begin{theorem}
\label{th:min1}
If $\mathbb{C}\subseteq\mathbb{A}$, then $\mathbb{C}\oplus\mathbb{B}\subseteq\mathbb{A}\oplus\mathbb{B}$.
\end{theorem}
\begin{proof}
Each element in $\mathbb{C}\oplus\mathbb{B}$ is given by $c+b$, where $c\in\mathbb{C}$ and $b\in\mathbb{B}$. 
Since $\mathbb{C}\subseteq\mathbb{A}$, then $c\in\mathbb{A}$. 
Thus, from the definition of Minkowski addition \eqref{eq:Minkowski}, we have  $(c+b)\in\mathbb{A}\oplus\mathbb{B}$.
\end{proof}
\begin{theorem}
\label{th:min2}
If $\mathbb{C}\subseteq\mathbb{A}$ and $\mathbb{D}\subseteq\mathbb{B}$, then $\mathbb{C}\oplus\mathbb{D}\subseteq\mathbb{A}\oplus\mathbb{B}$.
\end{theorem}
\begin{proof}
From the previous theorem, we have $\mathbb{C}\oplus\mathbb{B}\subseteq\mathbb{A}\oplus\mathbb{B}$ and  $\mathbb{C}\oplus\mathbb{D}\subseteq\mathbb{C}\oplus\mathbb{B}$, which together imply that  $\mathbb{C}\oplus\mathbb{D}\subseteq\mathbb{A}\oplus\mathbb{B}$.
\end{proof}
\begin{theorem}
\label{th:Emax}
    If $\mathbb{E}_{k}\subseteq\mathbb{E}\maxu$ and the system follows the \dtmpc\ policy $\pi(\cdot)$, then $\mathbb{E}_{i|k}\subseteq\ \mathbb{E}\maxu$ for 
    $i=k+1,\ldots,k+N$, where $\{z_{i|k}\}\oplus\mathbb{E}_{i|k}$ are the elements of the tube $\mathbb{T}_k$ of \dtmpc.
\end{theorem}
\begin{proof}
    By contradiction, if the theorem is not true, i.e., if there exists $\mathbb{E}_{i|k}$ 
    for $i = k+1, \ldots, k+N$, such that $\mathbb{E}_{i|k}\not\subseteq\mathbb{E}\maxu$ then there is at least one element 
    $\bm{e}_{i|k} = \bx_{i|k} - \bz_{i|k}$ corresponding to a possible realization 
    $\bx_{i|k}$ of the state at time step $i$, such that  $\bm{e}_{i|k} \in \mathbb{E}_{i|k}$, 
    but $\bm{e}_{i|k} \notin \mathbb{E}_{\max, i|k}$. 
    Therefore, for the corresponding external disturbances the state leaves tube $\mathbb{T}\maxu$,  
which is not possible because $\mathbb{T}\maxu$ is a robust positive invariant set for the policy $\pi(\cdot)$ according to assumption A1. 
\end{proof}
\begin{theorem}
\label{th:SRCI}
    The set $\Phi:=\{\{\bz_{i|k}\}\oplus \mathbb{E}_{i|k} 
    \big| \bz_{i|k}\in \mathbb{Z}_{\text{f}},\ \mathbb{E}_{i|k}\subseteq\mathbb{E}\maxu,\ i=k+1, \ldots, k+N\}$ 
    is a set robust positive invariant set for the \dtmpc\ inputs  $\pi(\bx_{i|k},\kappa(\bz_{i|k}),\bz_{i|k})$, for $i = k, \ldots , k+N-1$. (Note that according to \cite{TMPC} a set robust positive invariant set is defined as a set of sets. Since the terminal real state $x_{k|N+k}$ should be bounded, we need to constrain both the nominal state and the potential errors, thus a set of tubes (i.e., $\Phi$) is needed.)\\
\end{theorem}
\begin{proof}
    From assumption A3 it follows that if $\bz_k \in \mathbb{Z}_{\textrm{f}}$, then $\bz_{i|k} \in \mathbb{Z}_{\textrm{f}}$ for $i=k+1, \ldots,k+N$. 
    Moreover, from Theorem~\ref{th:Emax} if $\mathbb{E}_k \subseteq \mathbb{E}\maxu$, then $\mathbb{E}_{i|k}\subseteq\mathbb{E}\maxu$ for $i = k+1, \ldots, k+N$. 
    Therefore, from the definition of $\Phi$, we have $\{\bz_{i|k}\}\oplus\mathbb{E}_{i|k}\in\Phi$, 
    which proves the theorem.  
\end{proof}
\begin{theorem}
\label{th:feasbile}
If the sequence of nominal control inputs $\tilde{\bv}_k=[\bv^\top_{k|k},\bv^\top_{k+1|k},\dots,\bv^\top_{k+N-1|k}]^\top$ of \dtmpc\ is feasible for the admissible pair ($\bz_k$,$\bx_k$), i.e., using these inputs and  policy $\pi(\cdot)$, the controlled system satisfies $\bx_{i|k}\in\mathbb{X}$ for $i=k+1,\ldots,k+N$, 
then the shifted sequence of nominal control inputs  
$\tilde{\bv}_{k+1}=[\bv^\top_{k+1|k},\bv^\top_{k+2|k},\dots,\bv^\top_{k+N-1|k},\kappa^\top(\bz_{k+N|k})]^\top$ is feasible for the admissible pair $(\bz_{k+1},\bx_{k+1})$ under the \dtmpc\ policy $\pi(\cdot)$.%
\end{theorem}
\begin{proof}
    The nominal system is deterministic, and at time step $k+1$ for the first $N-1$ time steps  the same nominal control inputs  as for time step $k$ are applied. Thus one can write $\bz_{i|k+1}=\bz_{i|k}$ for 
    $i = k+1,\dots, k+N $. 
    The error set $\mathbb{E}_{k+1|k}$ contains all the possible errors of $\bx_{k+1|k}$ 
    with respect to $\bz_{k+1|k}$. At time step $ k+1$, however, $\mathbb{E}_{k+1|k+1}$ contains only one element of $\mathbb{E}_{k+1|k}$ based on the realized state under the \dtmpc\ policy. 
    Thus, we have $\mathbb{E}_{k+1|k+1}\subseteq\mathbb{E}_{k+1|k}$. 

    Since $\bz_{k+1|k+1}=\bz_{k+1|k}$ and $\mathbb{E}_{k+1|k+1}\subseteq\mathbb{E}_{k+1|k}$, 
    from Theorem~\ref{th:min2} we have 
    $\{\bz_{k+1|k+1}\}\oplus\mathbb{E}_{k+1|k+1}\subseteq \{\bz_{k+1|k}\}\oplus\mathbb{E}_{k+1|k}$. 
    From \eqref{eq:FIS_tube}-\eqref{eq:final_tube}, the next error set is derived based on 
    a mapping from all elements of the current error set added to the nominal state. 
    Based on the last two statements, Definition~\ref{def:D1}, and Theorem~\ref{th:min2}, we  
    conclude that $\mathbb{E}_{k+2|k+1}\subseteq\mathbb{E}_{k+2|k}$.  
    Similarly, in an iterative way it can be shown that $\mathbb{E}_{i|k+1}\subseteq\mathbb{E}_{i|k}$ for 
    $i= k+1,\dots,k+N$. 
    
    The sequence $\tilde{\bv}_k$ is assumed to be feasible for the admissible pair ($\bz_k$,$\bx_k$). Thus the elements 
    $\{\bz_{i|k}\}\oplus\mathbb{E}_{i|k}$ for $i = k+1,\ldots,k+N-1$ 
    belong to the admissible state set. 
    Now since $z_{i|k+1}=z_{i|k}$ and $\mathbb{E}_{i|k+1}\subseteq\mathbb{E}_{i|k}$ for $i =k+1,\dots,k+N$, based on Theorem~\ref{th:SRCI}, 
    the elements $\{\bz_{i|k+1}\}\oplus\mathbb{E}_{i|k+1}$ for $i = k+1,\ldots,k+N-1 $ 
    belong to the admissible state set. 
    Finally, since at time step $k+N$ 
    the \dtmpc\ input is determined based on the nominal control input $\kappa(\bz_{k+N|k})$ 
    (or equivalently, based on $\kappa(\bz_{k+N|k+1})$),  
    and since $\mathbb{E}_{k+N|k+1}\subseteq\mathbb{E}_{k+N|k}\subseteq\mathbb{E}\maxu$, 
    from Theorem~\ref{th:SRCI} we conclude that $\{\bz_{k+N|k+1}\}\oplus\mathbb{E}_{k+N|k+1}$ 
    belongs to the admissible state set.
\end{proof}
\begin{theorem}
\label{th:optimal_V}
    The optimal cost function $V^*(\cdot)$ in \eqref{eq:cost} is a Lyapunov function for system \eqref{eq:system} 
    (which implies stability).
\end{theorem}
\begin{proof}
    Given assumption A4, the optimal cost function $V^*(\cdot)$ is lower bounded by $\alpha_1(\cdot)$. 
    Moreover, an upper bound can be found for $V^*(\cdot)$, since it is continuous and finite for $\bx, \bz\in\mathbb{X}$. 
    The optimal sequence of nominal control inputs corresponding to the optimal cost $V^*(\bx_k,\bz_k)$ 
    is given by $\boldsymbol{\bv}_k^*=[{\bv^*}^\top_{k|k},{\bv^*}^\top_{k+1|k},\dots,{\bv^*}^\top_{k+N-1|k}]^\top$. 
    Based on Theorem~\ref{th:feasbile}, for the next control time step the sequence $\tilde{\bv}_{k+1}=[{\bv^*}^\top_{k+1|k},{\bv^*}^\top_{k+2|k+1},\ldots,{\bv^*}^\top_{k+N-1|k},\kappa^\top(\bz_{N|k})]^\top$ 
    is feasible for the admissible pair $(\bz_{k+1},\bx_{k+1})$ under the \dtmpc\ policy $\pi(\cdot)$.      
    The difference between $V^*(\bx_k,\bz_k)$ and the cost under sequence $\tilde{\bv}_{k+1}$ 
    is $l(\bz_{k|k},\bv_{k|k})+V_{\text{f}}(\bz_{k+N|k})-l(\bz_{k+N|k+1},\kappa(\bz_{k+N|k+1}))-V_{\text{f}}(f(\bz_{k+N|k+1},\kappa(\bz_{k+N|k}))$. 
    From assumption A3, this difference is positive, implying that the optimal cost $V^*(\bx_{k}, \bz_k)$ 
    is larger than the cost under $\tilde{\bv}_{k+1}$, which itself is larger than or equal to the optimal cost $V^*(\bx_{k+1},\bz_{k+1})$. Therefore, we have $V^*(\bx_{k+1},\bz_{k+1})<V^*(\bx_{k},\bz_{k})$.
\end{proof}
\section{Case stud{ies}}
\label{sec:case_study}

{In this section, we provide the results of two numerical case studies that have been designed for a SaR robot: 
\begin{itemize}
    \item 
    \textbf{Case study 1:} A \dtmpc\ controller is developed for a linear model of the robot, which is impacted by state-dependent disturbances. 
    With this experiment we showcase the unique capability of \dtmpc, in comparison with MPC and TMPC, to provide robustness 
    with respect to the disturbances, while improving the feasibility and reducing the conservativeness. The codes for this experiment have been published in \cite{code1}.
    \item 
    \textbf{Case study 2:} To showcase the performance of \dtmpc\ for nonlinear systems, a \dtmpc\ controller is developed for a nonlinear reference tracking problem for the robot. This problem has previously been 
    addressed in \cite{UnicycleNTMPC} using nonlinear TMPC. While the original nonlinear TMPC solution suffers from a high rise time, our results indicate a decreased rise time using \dtmpc, which allows the robot to safely move with a speed higher than the speed that is allowed by TMPC, without violating the constraints. 
    The code has been published in \cite{code2}.
\end{itemize}}

\subsection{Case study 1}
{\subsubsection{Simulation setup}}
\label{sec:setup_simulations}

We simulate a holonomic ground robot, for which the kinematics is described by the following discrete-time state-space representation:
\begin{equation}
\label{eq:holonomic_robot}
    \bx_{k+1}=
    \begin{bmatrix}
        1&0&T_s&0\\0&1&0&T_s\\0&0&1&0\\0&0&0&1
    \end{bmatrix}\bx_k+\begin{bmatrix}
        0&0\\0&0\\T_s&0\\0&T_s
    \end{bmatrix}\bu_k+ 
    \begin{bmatrix}
    \bm{w}_k \\
    0_{2\times 1}
    \end{bmatrix}
\end{equation}
where the state $\bx_k=[p_{x,k} , p_{y,k} , \nu_{x,k} , \nu_{y,k} ]^{\top}$ 
and control input $\bm{u}_k=[a_{x,k} , a_{y,k} ]^{\top}$ vectors include, respectively, 
the position and velocity, and the acceleration of the robot in the $x$ and $y$ directions.  
The discretization sampling time, $T_s$, is $0.1~$s,  and $\bm{w}_k$ 
is the 2-dimensional vector of the external disturbances that influence the position 
of the robot. In this case study, we assume that the disturbance affects position only (a similar assumption was used in \cite{UnicycleNTMPC}).
The set $\mathbb{W}(\bx_k)$ of external disturbances when the state of the robot is $\bx_k$ 
is a two-dimensional ellipse, with its major parallel to the direction of the movement 
of the robot. The magnitude of the disturbances is nonlinearly dependent on the state 
of the robot, i.e.:
\begin{subequations}
\label{eq:ellipse}
\begin{equation}
    r\maxu(\bx_k)=0.202\sqrt{\nu_{x,k}^2+\nu_{y,k}^2}+r_{\min}(\bx_k)
\end{equation}
 \begin{equation}
    r_{\min}(\bx_k)=0.225 \beta^2(\bx_k) \sqrt[4]{\nu_{x,k}^2+\nu_{y,k}^2}
\end{equation}   
\end{subequations}
with $r\maxu(\bx_k)$ and $r_{\min}(\bx_k)$ the length of the, respectively, major and minor of the ellipse when the state of the robot is $\bx_k$. 
In SaR missions, the external disturbances may vary in different parts of the environment, where the parameter $\beta(\bx_k)$, an indication of the slipperiness of the ground, 
allows us to include this in the simulations. 
We assume that the robot has perfect knowledge of its own state and of $\beta(\bx_k)$, but has no information on \eqref{eq:ellipse}. 
This assumption is made to make the learning procedure faster for the case study. In realistic scenarios, however,  understanding how ground conditions affect system behaviour is challenging, but not impossible \cite{DARPA_2022_winners}.

The constraints on the position of the robot vary in different simulations, 
while the constraints on the inputs and the velocity of the robot are always the following:
\begin{equation}\label{eq:constraints}
    |\nu_{x,k}|<2\frac{\textrm{m}}{\textrm{s}},\ |\nu_{y,k}|<2\frac{\textrm{m}}{\textrm{s}},
    \ |a_{x,k}|<5\frac{\textrm{m}}{\textrm{s}^2},\ |a_{y,k}|<5\frac{\textrm{m}}{\textrm{s}^2}
\end{equation}

The Table \ref{tab:math41} contains all addational mathematical notations required for this case study.

\begin{table}[]\caption{The following mathematical notation is used in section 4.1:}\resizebox{\textwidth}{!}{
\label{tab:math41}

\begin{tabular}{llll}
\hline
\multicolumn{1}{|l|}{The notation}                                                                                                                                        & \multicolumn{1}{l|}{The explenation}                                                                                                                                      & \multicolumn{1}{l|}{The notation} & \multicolumn{1}{l|}{The explenation}                             \\ \hline
$[A^\text{inequ}_{n_1\times n_2},b^\text{inequ}_{n_1\times 1}]$& A pair of matrices used to describe polytope in $n_2$ dimensions with $n_1$ inequalities                                                                                  & $Q_K$                             & Stage cost matrix of states of the ancillary control law         \\
$a_{x,k}$                                                                                                                                                                 & The current acceleration component parallel to x-axis                                                                                                                     & $Q_\kappa$                        & Stage cost matrix of states of the terminal control law          \\
$a_{y,k}$                                                                                                                                                                 & The current acceleration component parallel to y-axis                                                                                                                     & $R$                               & Stage cost matrix of inputs of the original cost function        \\
$c_{n,m}$                                                                                                                                                                 & \begin{tabular}[c]{@{}l@{}}
A constant parameter describing an ellipsoidal set of disturbances \\ where $n$ is an order number and $m$ denotes to the major or minor axis \end{tabular} & $R_K$                             & Stage cost matrix of inputs of the ancillary control law         \\
$F$                                                                                                                                                                       & Terminal cost matrix                                                                                                                                                      & $R_\kappa$                        & Stage cost matrix of inputs of the terminal control law          \\
$K$                                                                                                                                                                       & A state feedback matrix used in LQR                                                                                                                                       & $r\maxu(\bx_k)$                 & The major of the ellipsoidal disturbance set given current state \\
$m$                                                                                                                                                                         & meters                                                                                                                                                                    & $r_{\min}(\bx_k)$                 & The minor of the ellipsoidal disturbance set given current state \\
$\frac{\text{m}}{\text{s}}$                                                                                                                                               & meters per second                                                                                                                                                         & $T_s$                             & Sampling time                                                    \\
$\frac{\text{m}}{\text{s}^2}$                                                                                                                                             & meters per second square                                                                                                                                                  & $\bu_k$                           & The current velocity                                             \\
$p_{x,k}$                                                                                                                                                                 & The current position component parallel to x-axis                                                                                                                         & $\nu_{x,k}$                       & The current velocity component parallel to x-axis                \\
$p_{y,k}$                                                                                                                                                                 & The current position component parallel to y-axis                                                                                                                         & $\nu_{y,k}$                       & The current velocity component parallel to y-axis                \\
$Q$                                                                                                                                                                       & Stage cost matrix of states of the original cost function   
& 
$I_{n\times n}$                    & Identity matrix with dimension $n$  \\
\hline
\end{tabular}}
\end{table}

{\subsubsection{Formulating and implementing linear \dtmpc\ for path planning of a SaR robot}}

For path planning of the SaR robot in various simulated environments with external disturbances, we 
formulate an \dtmpc\ problem with a quadratic cost function given by:
\begin{equation}
\label{eq:nominal_quadratic_cost}
V(\bz_k,\bv_k) = \sum_{i=k}^{k+N-1}\left(\norm{\bz_{i|k}}_Q^2+\norm{\bm{v}_{i|k}}_R^2\right)+\norm{\bz_{k+N|k}}_F^2    
\end{equation}
where a prediction horizon of $N=5$ and the following cost matrices are considered:
\[Q=\text{diag}(100,100,1,1),  
\quad R=I_{2\times 2},\]
\[F=\begin{bmatrix} 805&0&-571&0\\0&805&0&-571\\-571&0&655&0\\0&-571&0&655 \end{bmatrix}\]
The most common choice for the ancillary control law is linear state feedback, such that 
the control input at time step $k$ is given by:
\begin{align}
\label{eq:linear_feedback}
    \bu_k=K\bm{e}_k+\bv_k,\ K=\begin{bmatrix}
        7,98&0&4.42&0\\0&7,98&0&4.42
    \end{bmatrix}
\end{align}
where $\bm{e}_k=\bx_k - \bz_k$, and for determining $K$ we use {linear quadratic regulator (LQR)} method 
with $Q_K=\text{diag}(100,100,0.1,0.1)$ and $ R_K=I_{2\times 2}$. 
These matrices are tuned, such that the ancillary control law more aggressively reduces the position. 
Consequently, the influence of the state-dependent external disturbances is more significantly reduced. 
The error $\bm{e}_k$ evolves according to:
\begin{equation}
\label{eq:linear_error}
    \bm{e}_{k+1} = (A+BK)\bm{e}_k+\begin{bmatrix}\bm{w}^{\top}_k | 0_{1\times 2}\end{bmatrix}^{\top}
\end{equation}
with $A$ and $B$ the $4\times 4$ and $4\times 2$ dynamic and input matrices given in \eqref{eq:holonomic_robot}. 
Thus the equivalent of \eqref{eq:2nd_tube} and \eqref{eq:policy_feasible} for this linear case, for $i=k,\ldots,k+N-1$, are respectively:
\begin{equation}
\label{eq:linear_shape}
    \Bar{\mathbb{E}}_{i+1|k} = (A+BK)\mathbb{E}_{i|k}
\end{equation}
\begin{equation}
\label{eq:lin_policy_feasible}
    K\mathbb{E}_{i|k}\oplus\{\bv_{i|k}\} \subseteq \mathbb{U}
\end{equation}
See Definition~\ref{def:D1} for how to compute \eqref{eq:linear_shape} and \eqref{eq:lin_policy_feasible}. 
Moreover, in order to reduce the computation time and because the external 
disturbance set changes slowly with the states in the proposed scenarios, we simplified \eqref{eq:FIS_tube}, such that it only 
takes into account the centroid of the tubes: 
\begin{equation}
\label{eq:simplifcation}
    \Bar{\mathbb{W}}_{i|k} = \text{FIS}\left(\bz_{i|k}+(A+BK)^{i-k}\boldsymbol{e}_k\right)
\end{equation}

The terminal admissible set $\mathbb{Z}_{\text{f}}$ for the nominal state is a positive invariant 
polytope, which is determined by assuming that the system follows an LQR control law $\kappa(\cdot)$ beyond the prediction horizon,   with $Q_{\kappa}
=\text{diag}(10,10,1,1)$ and $R_{\kappa}=I_{2\times 1}$. 
Then the terminal cost in \eqref{eq:nominal_quadratic_cost} is the limit of the cost when the system follows the LQR control  
law $\kappa(\cdot)$, when $k \rightarrow \infty$, and can be determined by solving the following Riccati equation,  using, e.g., Matlab dlyap command \cite{dlyap}: 

\begin{equation}\label{eq:ricatti}
    F = 2((A+BK)^TF(A+BK)+Q_{\kappa}+K^TR_{\kappa}K)
\end{equation}
For the details on the derivation of \eqref{eq:ricatti}), we refer the reader to  \cite{MPCbook}. 
Control  policy  $\kappa(\cdot)$ should be designed, such that Assumption~\ref{as:kappa} holds.

{\subsubsection{Learning the external disturbance sets using a FIS}}

In order to approximate $\mathbb{W}(\bx_k)$ by $\Bar{\mathbb{W}}_k$, 
polyhedral sets defined by {$\{[x,y]^\top:A^{\textrm{ineq}}_{8\times2}[x,y]^\top
\leq \bm{b}^{\textrm{ineq}}_{8\times1}\}$} are considered. The advantage of using polytope sets is that, unlike ellipses, the Minkowski addition can be computed for such sets without approximation and the shape of the set remains a polytope after using Minkowski additions. 
Moreover, with polytope sets real-life disturbance sets of different shapes can be approximated. 
The main downside of using polytope sets is that the problem can become computationally expensive for large  dimensions. In such cases, ellipsoids can be used instead. 
Thus, in order to simplify the fuzzy inference systems that should estimate $\Bar{\mathbb{W}}_k$, 
we consider two TSK FISs for approximating the major and minor (i.e., \eqref{eq:ellipse}) of the disturbance sets per time step $k$, with their output functions defined by:
\begin{align}
\label{eq:quad}
    r_m(\bx_k)=c_{1,m}\nu_k^2+c_{2,m}\beta^2(\bx_k)+
    c_{3,m}\nu_k\beta(\bx_k)+c_{4,m}\nu_k+c_{5,m}\beta(\bx_k)+ c_{6,m} 
\end{align}
with $\nu_k^2 = \nu_{x,k}^2 + \nu_{y,k}^2$, $m\in\{\max,\min\}$, and 
$\bm{\theta}_{m,k}=[c_{1,m}\ \ldots\ c_{6,m}]^\top$. 
A polytope with a fixed number of edges that encounters the resulting ellipse 
is then determined and used as the disturbance set at time step $k$.

We generated $1240$ input-output pairs, based on the ground truth ellipsoids that were designed for the case study (see \eqref{eq:ellipse}). 
This data set was divided into a training and a validation set, with a proportion of $660:580$. 
For the test set, we generated a large data set with $100000$ pairs of input-output using \eqref{eq:ellipse} in order 
to extensively test the trained FISs. 
Given the training and the validation errors, the best results were achieved using 
$5$ fuzzy sets per input, which resulted in $25$ fuzzy rules with  outputs that are described by \eqref{eq:quad}. 
Note that compared to different expressions for the output of the rules, the expression given by \eqref{eq:quad} 
did not result in under or over-fitting.
GA was used to train the FISs by minimizing the mean square error, 
with an additional penalty for negative errors. 

\begin{table*}
\caption{A summary of the results for the trained FIS that estimate the lengths of the major and minor of the disturbance ellipse sets using the test data, where the errors are normalized via dividing by the maximum ground truth value.}
\label{tab:FIS_results}
\centering
\begin{tabular}{|l|l|l|}
\hline
&
FIS that estimates the major & 
FIS that estimates the minor
\\ \hline
Number (in \%) of cases where errors were negative & $2.64$ & $3.01$ \\ \hline
Mean value of negative errors & $7.8\times10^{-4}$ & $1.93\times10^{-3}$ \\ \hline 
Mean value of positive errors & $1.02\times 10^{-2}$ & $2.32\times 10^{-2}$ \\ \hline 
Maximum value of negative errors &  $5.89\times10^{-3}$ & $2.56\times 10^{-2}$ \\ \hline 
Maximum value of positive errors & $5.94\times 10^{-2}$ & $2.45\times 10^{-1}$ \\ \hline
\end{tabular}
\end{table*}
\begin{figure}
    \centering
    \includegraphics[width=1\textwidth]{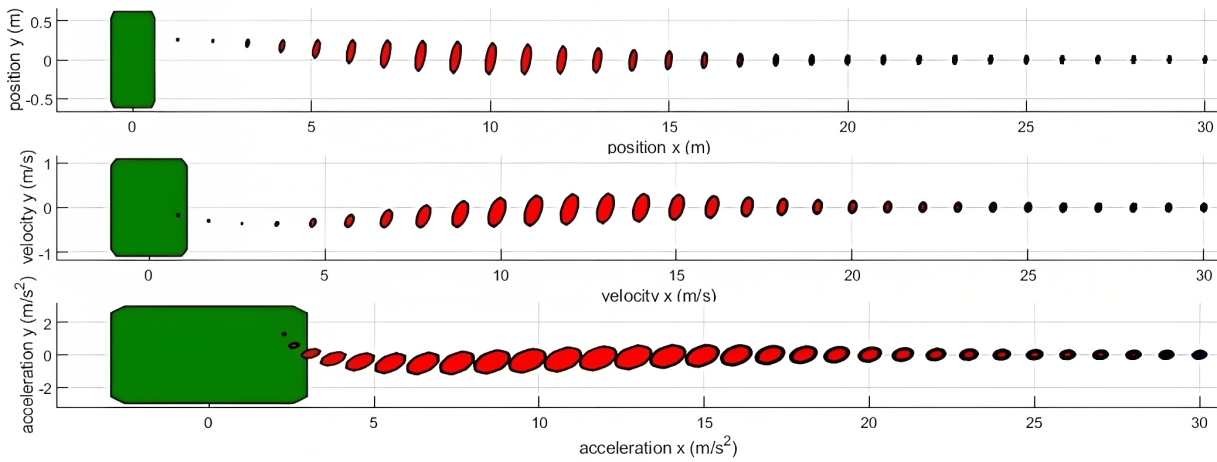}
    \caption{The projections of the tube of TMPC (shown in green) and the tubes of \dtmpc\ generated by the FIS 
    (shown in blue), compared with the ground truth admissible sets of disturbances 
    (shown in red), for a prediction horizon of $30$. To better illustrate the time evolution of the shape and size of the tubes, the projections, on the first two plots, were translated in such a way that nominal position and velocity are ordered as a horizontal line i.e $z_{i|0}=[i,0,i,0]^\top$. On the last plot, the projections were translated in such a way that $Kz_{i|0}=[i,0]^\top$}.
    \label{fig:err}
\end{figure}

The results obtained for the test set for the two FIS that estimate the lengths of the major and minor of the ellipse sets for the external disturbances are summarized in Table~\ref{tab:FIS_results}, 
where the error values have been normalized. 
In general, both FIS was able to approximate the ground truth with mostly positive-valued errors. 
Negative-valued errors could completely be eliminated by including {penalties} for small 
positive errors, in addition to the negative ones. This, however, may increase the overall error, 
and our simulation results showed that it is not necessary. 

In \dtmpc, the tube changes over time. Thus, we checked for a sample example  
of how the tube evolves during the optimization when a larger horizon of $30$ is used. 
Since the tube is $4$-dimensional, it is impossible to present it in one plot. 
Thus, Figure~\ref{fig:err} shows the projection of the tube on the planes of positions and velocities. Additionally,
the tube multiplied by the ancillary control gain $K$ is shown that is basically the tube of control input i.e. acceleration.
All these plots also show the projection of the tube $\mathbb{T}$ of TMPC that was designed based on \cite{Calculting_tube}. 
The tube of \dtmpc\ evolves until there is convergence in all dimensions of the state,  
whereas the convergence is never close to the tube of TMPC. 
The evolution of the tube can result in both the growth and shrinkage of the tube. 
For instance, once the system approaches the origin and slows down, the error 
sets become smaller as a result of the system slowing down. 
What is also important, we plotted both tubes generated while \dtmpc\ was using FIS and ground truth functions and it can be observed that they are almost overlapping. It shows that the current approximation is sufficient. 

{\subsubsection{Comparison: MPC, TMPC, \dtmpc}}
\label{sec:case_study_scenarios}

In all the given simulations, regular MPC (formulated by \eqref{eq:MPC}), TMPC (formulated by \eqref{eq:TMPC}), and \dtmpc\ (formulated by \eqref{eq:SDDTMPC}) were   
implemented and their performance, with respect to minimizing the cost 
function, satisfying the hard constraints, and computation time 
{were} compared. { $\ominus$ symbolizes a Minkowski  difference.}
The controllers were implemented in MATLAB version 2022b  
running on a PC with a 11th Gen Intel(R) Core(TM) i7-1185G7 CPU 
with a clock speed of 3 GHz.
In order to solve the optimization problems, the quadratic programming solver (Matlab's Optimization toolbox) 
was used for MPC and TMPC, whereas for \dtmpc, the particle swarm (PS) algorithm {\cite{ParticleSwarm}}
from the Global optimization toolbox of Matlab was implemented. In this toolbox, only PS and GA could consistently 
minimize the cost function, whereas PS usually converged faster. 
Unless stated otherwise, the optimization procedure for SDD-TMPC
was terminated after 60 iterations which took about 4 minutes to run, which has been shown in the simulations to 
provide a balanced trade-off between the computation time and the accuracy of the solutions.
{In comparison, MPC and TMPC  solved their quadratic problems in milliseconds. }
The initial nominal state was assumed to be equal to the real state, $\beta(\bx_k)$ was assumed to be $1$, and the horizon was assumed to be 5 unless otherwise stated. Terminal constraint and cost were not used in simulations 2 (possible, but the horizon would be too high, which would slow down the computation too much) and 3 (impossible, because the goal is outside the feasible set, so there is no way to design the terminal set).

\begin{subequations}\label{eq:MPC}
    \begin{equation}
     V^*(\bx_k) = \min_{\Tilde{\bx}_k,\Tilde{\bu}_k}\sum_{i=k}^{k+N-1}\left(\norm{\bx_{i|k}}_Q^2+\norm{\bm{u}_{i|k}}_R^2\right)+\norm{\bx_{k+N|k}}_F^2 
\end{equation}
\begin{equation*}
     \text{ s.t. for $i = k+1, \ldots, k+N$: }
\end{equation*}
\begin{equation}
    \bx_{i|k}=A\bx_{i-1|k}+B\bu_{i-1|k}  
\end{equation}
\begin{equation}
     \bx_{i|k}\in \mathbb{X}, \; \bu_{i|k} \in \mathbb{U}, \; \bx_{N|k} \in \mathbb{Z}_\text{f}
\end{equation}
\end{subequations}

\begin{subequations}\label{eq:TMPC}
    \begin{equation}
     V^*(\bz_k) = \min_{\Tilde{\bz}_k,\Tilde{\bv}_k}\sum_{i=k}^{k+N-1}\left(\norm{\bz_{i|k}}_Q^2+\norm{\bm{v}_{i|k}}_R^2\right)+\norm{\bz_{k+N|k}}_F^2 
\end{equation}
\begin{equation*}
     \text{ s.t. for $i = k+1, \ldots, k+N$: }
\end{equation*}
\begin{equation}
    \bz_{i|k}=A\bz_{i-1|k}+B\bv_{i-1|k}  
\end{equation}
\begin{equation}
     \bz_{i|k}\in \mathbb{X}\ominus\mathbb{E}_{\text{max}}, \; \bv_{i|k} \in \mathbb{U}\ominus\mathbb{E}_{\text{max}}, \; \bz_{N|k} \in \mathbb{Z}_\text{f}
\end{equation}
\end{subequations}

\begin{subequations}
\label{eq:SDDTMPC}
\begin{equation}
     V^*(\bx_k,\bz_k) = \min_{\Tilde{\bz}_k,\Tilde{\bv}_k,\mathbb{T}_k}\sum_{i=k}^{k+N-1}\left(\norm{\bz_{i|k}}_Q^2+\norm{\bm{v}_{i|k}}_R^2\right)+\norm{\bz_{k+N|k}}_F^2 
\end{equation}
\begin{equation*}
     \text{ s.t. for $i = k+1, \ldots, k+N$: }
\end{equation*}
\begin{equation}
    \bz_{i|k}=A\bz_{i-1|k}+B\bv_{i-1|k}  
\end{equation}
\begin{equation}
  \label{eq:tubeTk_lin}
     \mathbb{T}_k = \{\{z_{k+1|k}\}\oplus\mathbb{E}_{k+1|k}, \ldots, \{z_{N|k}\}\oplus\mathbb{E}_{k+N|k} \}    
 \end{equation}
\begin{equation}
    \Bar{\mathbb{E}}_{i|k} =(A+BK)\mathbb{E}_{i-1|k}
\end{equation}
\begin{equation}
    \Bar{\mathbb{W}}_{i-1|k} = \text{FIS}\left(\bz_{i|k}+(A+BK)^{i-k}(\boldsymbol{x}_k-\boldsymbol{z}_k)\right)
\end{equation}
\begin{equation}
    \mathbb{E}_{{i|k}} = \Bar{\mathbb{E}}_{i|k}\oplus\Bar{\mathbb{W}}_{i-1|k}
\end{equation}
\begin{equation}
   \mathbb{E}_k=\{\bx_k-\bz_k\}
\end{equation}
\begin{equation}
     \{\bz_{i|k}\}\oplus\mathbb{E}_{i|k} \subseteq \mathbb{X} 
\end{equation}
\begin{equation}
    K\mathbb{E}_{i|k}\oplus\{\bv_{i|k}\} \subseteq \mathbb{U}
\end{equation}
\begin{equation}
  \bz_{N|k} \in \mathbb{Z}_\text{f}
\end{equation}
\end{subequations}

To compare the convergence of the optimization algorithms, a sample problem was considered where the robot should reach the zero state starting from the state $[0.35,0.65,0,0]^\top$.  The table~\ref{table:convergence} shows the optimal costs for MPC, TMPC, and \dtmpc.    Since \dtmpc\ solves a nonconvex, nonlinear optimization problem, it may not be possible to find a global minimum. Therefore, the changes in the value of the cost by iteration of the PS algorithm are given in the table (the number of iterations is given in parentheses). The rate of these changes depends strongly on how constrained the problem is (compare the third and sixth rows of the table).  As a warm start for \dtmpc, the solution of TMPC is used, and in case TMPC is infeasible, the shifted trajectory from Theorem~\ref{th:feasbile} is used.  As expected, the cost values corresponding to TMPC and MPC, respectively, are  upper and lower bounds for the costs of \dtmpc.

\begin{table}
\caption{Comparing the cost values for MPC, TMPC, and different numbers of iterations of PS,  
which solves \dtmpc\ when the robot should reach the zero state from  
state $[0.35,0.65,0,0]^\top$ (no stage position constraints).}. 
\label{table:convergence}
\begin{tabular}{|lllllll|}
\hline
\multicolumn{7}{|c|}{Without terminal constraint and terminal cost} \\ \hline
\multicolumn{1}{|l|}{MPC} & \multicolumn{1}{l|}{TMPC} & \multicolumn{1}{l|}{PS(40)} & \multicolumn{1}{l|}{PS(80)} & \multicolumn{1}{l|}{PS(120)} & \multicolumn{1}{l|}{PS(160)} & PS(200) \\ \hline
\multicolumn{1}{|l|}{339.23} & \multicolumn{1}{l|}{379.76} & \multicolumn{1}{l|}{340} & \multicolumn{1}{l|}{339.6} & \multicolumn{1}{l|}{339.6} & \multicolumn{1}{l|}{339.6} & 339.6 \\ \hline
\multicolumn{7}{|c|}{With terminal constraint and terminal cost} \\ \hline
\multicolumn{1}{|l|}{MPC} & \multicolumn{1}{l|}{TMPC} & \multicolumn{1}{l|}{PS(40)} & \multicolumn{1}{l|}{PS(80)} & \multicolumn{1}{l|}{PS(120)} & \multicolumn{1}{l|}{PS(160)} & PS(200) \\ \hline
\multicolumn{1}{|l|}{402.91} & \multicolumn{1}{l|}{781.32} & \multicolumn{1}{l|}{667.05} & \multicolumn{1}{l|}{597.10} & \multicolumn{1}{l|}{540.48} & \multicolumn{1}{l|}{531.28} & 528.31 \\ \hline
\end{tabular}
\end{table}

The following $4$ scenarios were designed and simulated for situations 
relevant to SaR missions, each  
assessing one key property of \dtmpc, compared to MPC and TMPC. 
The scenarios were simulated several times to show the behavior of the controllers under different disturbances. 

{\paragraph{Scenario 1: Closely following a reference path}}

The main aim of the comparison among MPC, TMPC, and \dtmpc\ in scenario~1 is to 
assess how fast a robot that is controlled by each of these approaches 
reaches a given destination, following a reference path (i.e., an ordered set of given positions). 
As soon as the (nominal) state of the robot reaches a distance of $0.3~$m from its corresponding reference position, 
the next reference position from the path is sent to the controller as a new reference.  
Scenario~1 resembles a real-life SaR situation when a robot should precisely travel 
along a given path in {the} presence of external disturbances (e.g., to avoid hazards that exist in the close vicinity 
of the robot) that is determined via a higher-level controller. 
Scenario~1 was simulated 3 times, considering the following situations:
\begin{enumerate}
    \item The system was not affected by external disturbances.
    \item The external disturbances helped in attracting the system to the current reference position.
    \item The external disturbances resulted in repelling the system from the current reference position. 
\end{enumerate}
The entire mission time of the robot for scenario~1 for the 3 MPC-based controllers 
is represented in Table \ref{tab:s1}. 
Based on the results, for  MPC the mission time is strongly affected by 
the external disturbances (see the significant variations in the mission time {in} different cases), 
whereas TMPC and \dtmpc\  provide more consistent results independent of the realized disturbances. 
TMPC needs more time than \dtmpc\  to finish the mission. In fact, to keep the states  
within the tube of TMPC (determined for the worst disturbance case), the nominal acceleration and speed 
of the robot have to remain within $60\%$ and $50\%$ of their maximum allowed values.
\begin{table}[]
\caption{The mission time required by the robot in scenario~1  to closely track the entire 
reference path.}
\label{tab:s1}
\begin{tabular}{|l|lll|lll|}
\hline
 & \multicolumn{3}{c|}{Mission time {[}s{]} without terminal cost and constraint} & \multicolumn{3}{l|}{Mission time {[}s{]} with terminal cost and constraint} \\ \hline
 & \multicolumn{1}{l|}{MPC} & \multicolumn{1}{l|}{TMPC} & SDD-TMPC & \multicolumn{1}{l|}{MPC} & \multicolumn{1}{l|}{TMPC} & SDDT-MPC \\ \hline
Case 1 & \multicolumn{1}{l|}{10.8} & \multicolumn{1}{l|}{12.9} & 10.7 & \multicolumn{1}{l|}{12.5} & \multicolumn{1}{l|}{13.3} & 11.5 \\ \hline
Case 2 & \multicolumn{1}{l|}{7.8} & \multicolumn{1}{l|}{12.9} & 10.7 & \multicolumn{1}{l|}{8.6} & \multicolumn{1}{l|}{13.3} & 11.3 \\ \hline
Case 3 & \multicolumn{1}{l|}{18} & \multicolumn{1}{l|}{12.9} & 10.7 & \multicolumn{1}{l|}{19.8} & \multicolumn{1}{l|}{13.3} & 11.4 \\ \hline
\end{tabular}
\end{table}

{\paragraph{Scenario 2: Mission planning in the vicinity of obstacles}}

SaR robots should often move close to obstacles. 
In such cases, the trajectory of the robot must guarantee that no crashes occur 
with the obstacles. However, overly conservative decisions may significantly slow down the mission. 
In scenario~2, the robot has to move horizontally, from position $[2,0.01]^\top$ to position $[8.5,0.01]^\top$.
We suppose that if the vertical position of the robot falls below $p_y=0$, a crash with an obstacle (e.g., a wall 
that is extended across $p_y=0$) has occurred. 
Moreover, the value of $\beta(\bx_k)$ changes with the coordinates based on the following equation: 
\begin{equation}
\label{eq:s2b}
    \beta(\bx_k)=\frac{1}{|5-p_{x,k}|+1}
\end{equation}
Scenario~2 was simulated 3 times, considering the following situations:
\begin{enumerate}
    \item The system was not affected by external disturbances.
    \item The system was pushed upwards by external disturbances. 
    \item The system was pushed downwards by external disturbances.
\end{enumerate}
The results of the simulations for scenario~2 when 
MPC, TMPC, and \dtmpc\ are used to steer the robot, are  shown  in Figures~\ref{fig:s2}. 
 MPC steers the robot such that it moves close to the wall. However, for the third simulation when the system 
 was pushed downwards due to external disturbances, the robot collided into the wall.  
 Based on Figure~\ref{fig:s2}, TMPC assures a safe, collision-free trajectory for the robot,  
 while the robot moves relatively far away from the wall. 
 The initial nominal state of the robot with TMPC was far away from the wall, otherwise the initial 
 tube of TMPC would collide with the wall.
 Figure~\ref{fig:s2} also shows the nominal and the realized trajectories of the robot when it is  
 controlled via \dtmpc. While the robot moves much closer to the wall compared to TMPC,  
 it never crashes into it. 
 Note that from \eqref{eq:optimization}, the nominal states of \dtmpc\ 
 are determined based on an estimation of the external disturbances. 
Thus, the nominal trajectories vary when the disturbances do. In other words, the more slippery the floor, the more careful the actions chosen by \dtmpc.
\begin{figure}
    \centering
    \includegraphics[width=0.9\textwidth]{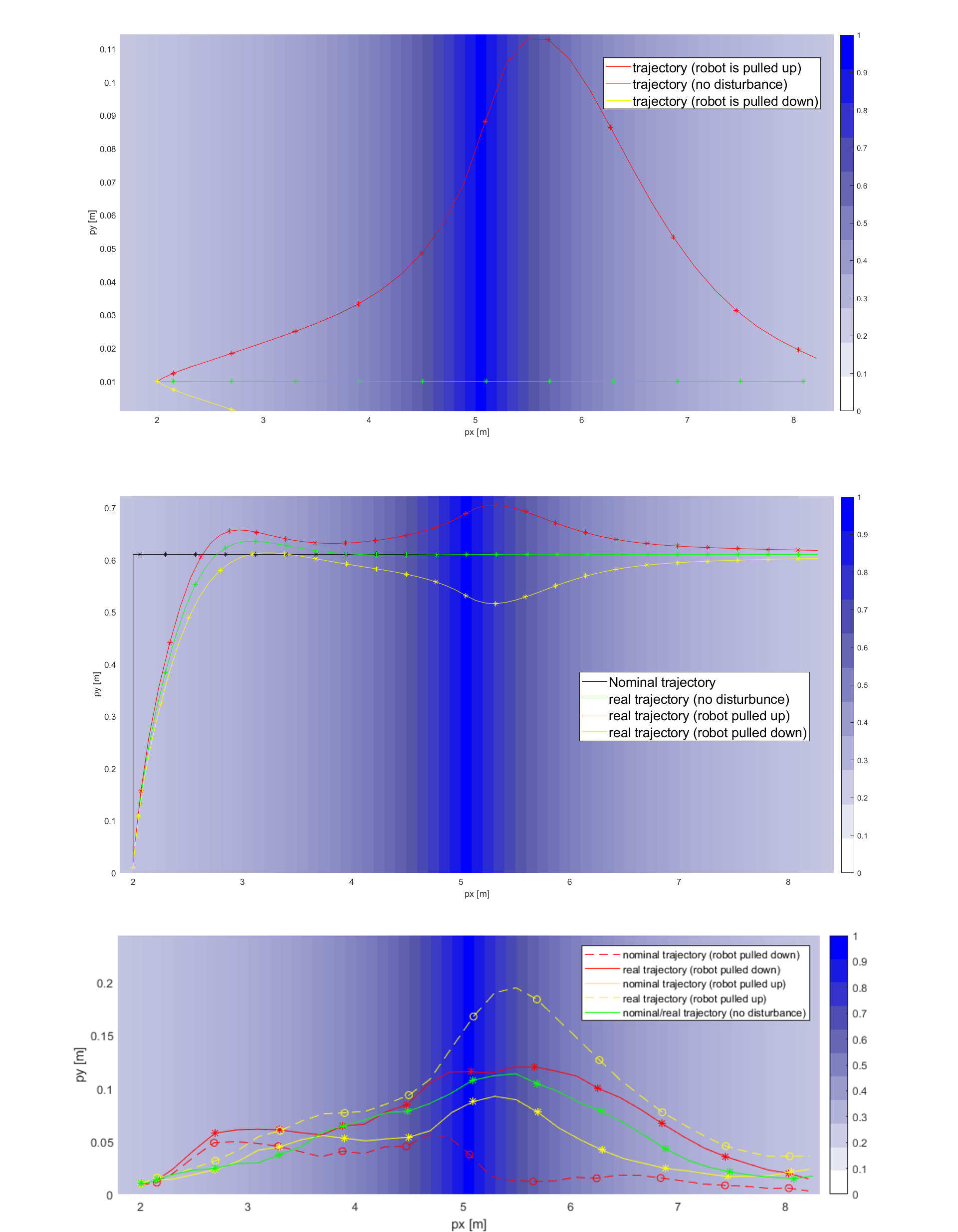}
    \caption{Scenario~2: Trajectories of the robot position when controlled by MPC (top plot), TMPC (middle plot) and \dtmpc (bottom plot), in case 1 (in green), 2 (in red), and 3 (in yellow). Note that the spectrum of blue corresponds to the values of $\beta(\bx_k)$ for all realized states $\bx_k$ of the robot during the simulation. In the middle plot, the black trajectory corresponds to the nominal case. 
    For the bottom plot, the corresponding nominal and real trajectories are shown in the same color but with solid and dashed lines respectively.}
    \label{fig:s2}
\end{figure}

\begin{figure}
    \centering
    \includegraphics[width=1\textwidth]{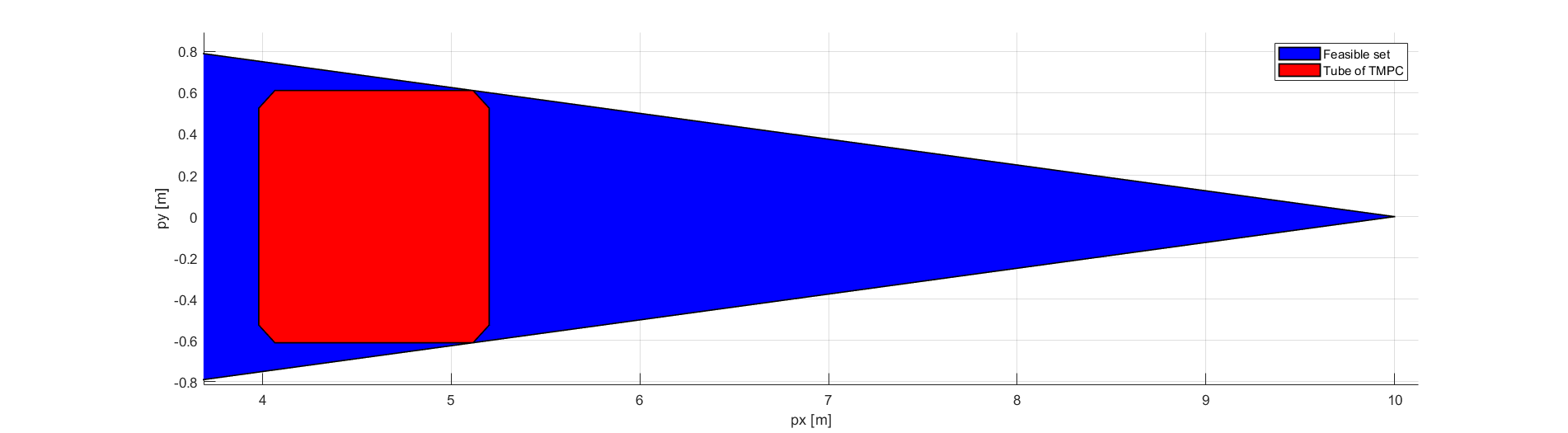}
    \caption{Scenario~3: The feasible state set $\mathbb{X}$ (in blue) and the tube of TMPC (in red).}
    \label{fig:3_Sets}
\end{figure}

\begin{figure}
    \centering
    \includegraphics[width=0.85\textwidth]{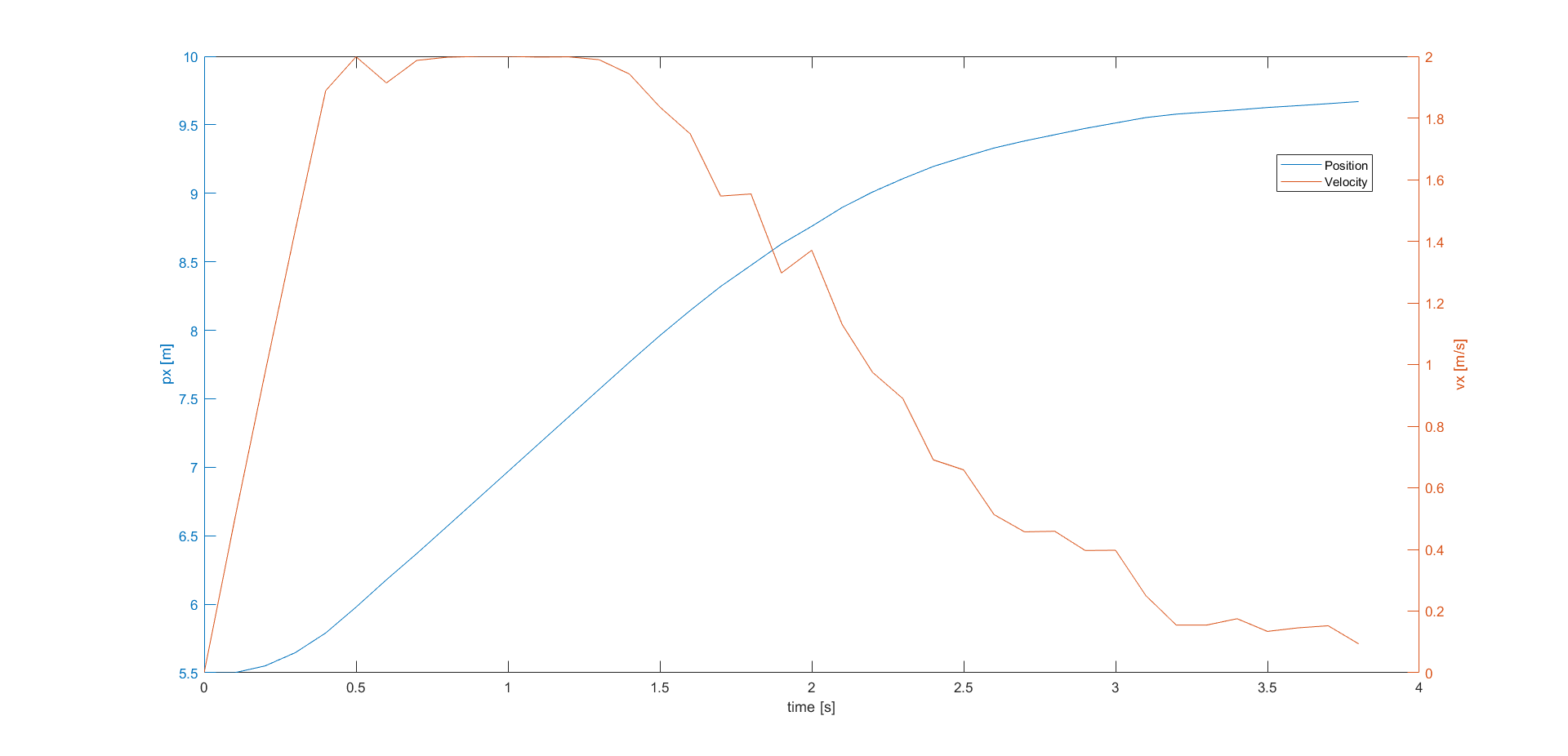}
    \caption{Scenario~3: The evolution of the position $p_x$ (blue) and the velocity $\nu_x$ (orange) of the robot, when \dtmpc\ is used and no external disturbances exist.}
    \label{fig:3_evolution}
\end{figure}

\begin{figure}
    \centering
    \includegraphics[width=0.85\textwidth]{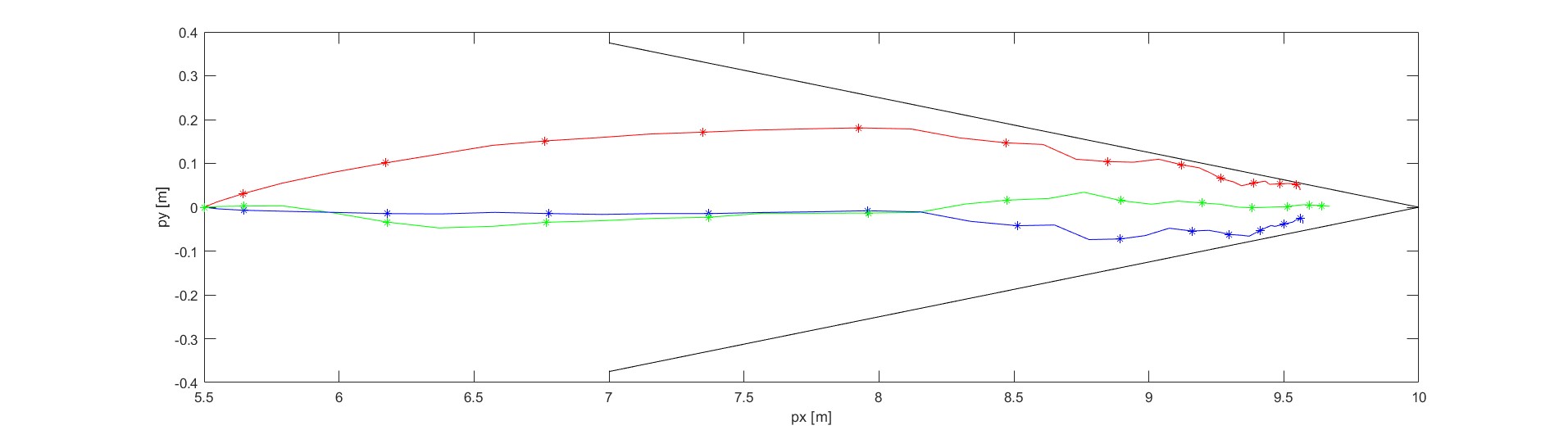}
    \caption{Scenario~3: Comparison of the trajectories of the position of the robot using \dtmpc, 
    when no external disturbances affect the robot (green curve) and when external disturbances push the robot upwards 
    relative to its speed {(the blue curve for the nominal trajectory and the orange curve for the realized trajectory)}. Note that the black lines represent the boundaries of the corridor.}
    \label{fig:3_trajectory}
\end{figure}

{\paragraph{Scenario 3: Approaching an unreachable target}}

In SaR missions, there are situations when the robot should move in very narrow 
corridors without crashing into the walls. 
In such cases, TMPC may become infeasible, due to its conservativeness. 
Thus scenario~3 is designed such that the problem is infeasible for TMPC, and 
we assess whether \dtmpc\ is able to return a solution for the problem. 
In scenario~3, the robot starts at position $[5.5,0]^\top $ with a zero initial speed, and has to reach 
position $[11,0]^\top$. The following hard constraints on the position of the robot 
hold for all time steps $k$ during the simulation, where these constraints imply 
avoiding crashes into the walls in a corridor that is narrowing down:
\begin{equation}
\label{eq:s3c}
    p_{x,k}+8p_{y,k}<10;\quad p_{x,k}-8p_{y,k}<10
\end{equation}
Note that all coordinates are given in m.  
TMPC is already infeasible at the initial position of the robot, 
since the farthest position for which TMPC can still generate a feasible tube 
corresponds to $p_x= 4.59$~m, where the width of the corridor is about $0.68$~m (see Figure~\ref{fig:3_Sets}, 
which shows the corridor and the tube of TMPC). 
The aim is to find out how far can the robot move forward inside this corridor until the \dtmpc\ problem becomes infeasible. 

Scenario~3 was simulated twice, 
considering the following situations: 
\begin{enumerate}
    \item The system was not affected by external disturbances.
    \item The system was pushed upwards (relative to the speed value) via external disturbances. 
\end{enumerate}
Note that in both cases, the upper bounds of potential external disturbances were given to TMPC, 
whereas \dtmpc\ used its FIS to estimate the upcoming disturbances. 
The prediction horizon used for this scenario was tuned to $6$ (i.e., larger than the previous scenarios), 
which was because in this scenario there is no terminal constraint as the target of the robot is outside of its feasible set. 
Thus, the alternative approach is to increase the size of the prediction horizon in order to reduce the risks of collision. 
With a prediction horizon smaller than $6$,  due to being short-sighted, the simulations showed that 
the robot accelerated rapidly and did not have time 
to slow down in time in order to avoid a crash into the walls when the corridor narrowed down.

In Figure~\ref{fig:3_evolution}, the evolution of the position and the velocity of the robot for the first situation 
(i.e., when no external disturbances exist) are shown, where \dtmpc\ is used to steer the robot. 
From this figure, since there is no wall in a close neighborhood of the robot, it first moves more freely (faster), 
and then slows down in time (thanks to the larger prediction horizon) as the corridor becomes narrower. 
In fact, with \dtmpc\ the robot can move further through the corridor, compared to when TMPC is used. 
In case external disturbances exist and push the robot upwards relative to its speed value, 
a similar pattern of behavior for the robot is observed, although the increase in the speed will be less significant 
than the case without external disturbances. 
Figure~\ref{fig:3_trajectory} shows the trajectories for the position of the robot for both cases, i.e., with and without 
external disturbances. 
By comparing the trajectories, we observe how \dtmpc\ changes the nominal vertical position (in presence of external disturbances),  
which, in a vacuum, raises the cost (because of deviating from the reference trajectory). 
However, \dtmpc\ makes this decision for the nominal trajectory of the robot, because it foresees that 
its ancillary control input will cancel out the impact of the external disturbances as much as possible 
for the realized trajectory $p_{y,k}$ of the robot. Such behavior cannot be obtained using lar MPC or TMPC.

{\paragraph{Scenario 4: Ability to make high-level optimal decisions}}
\label{sec:scenario4}

Next to making robust decisions, \dtmpc\ is expected to improve the performance by making high-level optimal decisions. 
In order to assess this, in scenario~$4$ we consider a case where the robot has to move from 
position $[2.3,3]^\top$ 
to position $[3.25,3.05]^\top$, with an obstacle positioned on the way of the robot (see Figure~\ref{fig:scenario_4}). 
We suppose that the robot has the initial horizontal speed of $\nu_x=1~\frac{\textrm{m}}{\textrm{s}}$, 
and should select an initial vertical speed of either $\nu_y=1.2~\frac{\textrm{m}}{\textrm{s}}$ 
(i.e., moving upwards) or $\nu_y=-1.2~\frac{\textrm{m}}{\textrm{s}}$ 
(i.e., moving downwards). 
When the robot moves upwards, the slipperiness coefficient is larger 
and thus, the robot is prone to larger disturbances, compared to when it moves downwards. 
Therefore, when the robot moves upwards, the larger external disturbances will result in a 
longer realized path (due to an increased distance from the obstacle) for the robot compared to when the robot moves downwards. 
It is trivial that since MPC and TMPC do not take into consideration the dynamics of the external disturbances, 
they estimate the less cost when $\nu_y$ is positive, whereas \dtmpc\ returns a lower cost ($772$) for the 
trajectory that corresponds to the initial vertical speed of $\nu_y=-1~\frac{\textrm{m}}{\textrm{s}}$ 
(since this trajectory is exposed to {fewer} external disturbances),  
compared to the other trajectory that is exposed to larger disturbances (estimated cost of \dtmpc\ is $917$).  In this scenario, we did not stop the optimization early.
Therefore, \dtmpc\ potentially performs better than MPC and TMPC (in terms of the realized cost) in complex environments that are 
prone to state-dependent external disturbances. 

\begin{figure}
    \centering
    \includegraphics[width=0.95\textwidth]{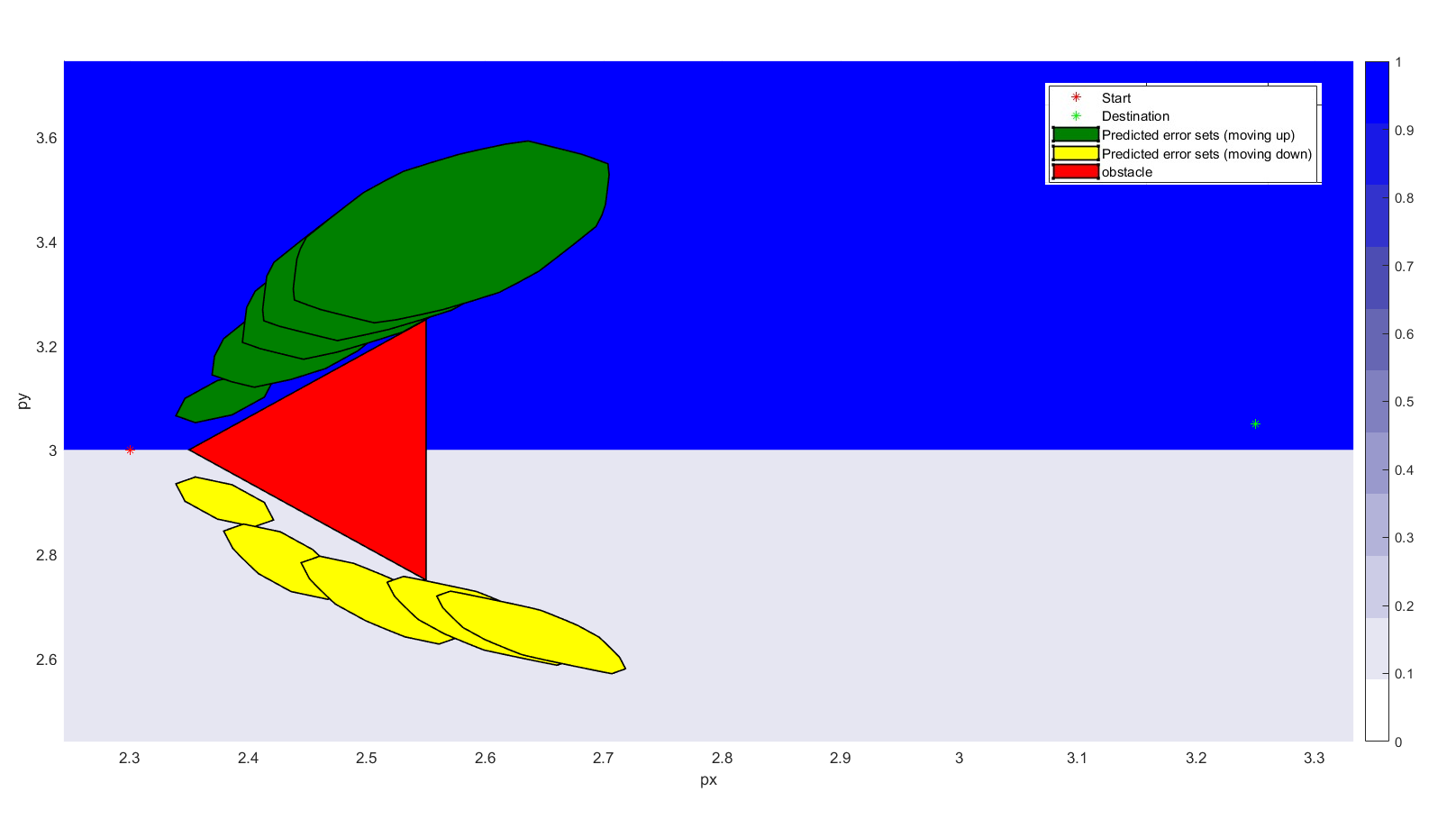}
    \caption{Scenario~4: The robot has to move from position $[2.3,3]^\top$ to position $[3.25,3.05]^\top$, 
    while avoiding an obstacle (illustrated by the triangular red shape). The trajectories of the position projections of the tubes for two cases were shown in green (above the obstacle) and yellow (below the obstacle) colors.
}
    \label{fig:scenario_4}
\end{figure}

{\subsection{Case study 2}
\subsubsection{Simulation setup}

The problem presented in \cite{UnicycleNTMPC} compared two controllers: nonlinear TMPC and nonlinear robust MPC. We replicated the experiment setup and the implementation of nonlinear TMPC, utilizing the same numerical values. The results will show that while using \dtmpc, the robot was able to achieve better performance by using less conservative control input (i.e., the robot moved faster) and still achieve a robust solution. The only detected differences to our knowledge are: 
\begin{itemize}
    \item Disturbance generator:  The authors presented solely the boundaries without elaboration on the random generation or publication of code.
    \item Solver: We used the same algorithm (interior point), but from Matlab's optimization toolbox. However, due to the nonconvex nature of the optimization problem, we cannot ensure that we will identify the exact same solutions.
    \item Tube initialization: In \cite{UnicycleNTMPC}, it was noted that the solver had the freedom to choose the initial nominal position at each time step, provided that the measured actual position remained within the tube. However, there is no mention of the starting nominal direction. We permitted the solver to freely choose the direction, but it remains uncertain if this was the intended design.
    \item Discretization method: In \cite{UnicycleNTMPC}, the authors studied a continuous system that was discretized using the ICLOCS toolbox \cite{ICLOCS2}. We employ Matlab's c2d function, using Zero-Order Hold (ZOH) \cite{ZOH} and Runge-Kutta 4 algorithm to discretize linear and nonlinear systems, respectively. 
    \item Hardware: We utilized the same computer that was employed in case study~1.
\end{itemize}

In the simulation, two unicycle robots are present: a virtual leader robot, which is assumed to remain unaffected by 
the external disturbances, and a follower robot, which should be controlled and is affected by the disturbances.  
A unicycle robot is a circular robot that is steered at time instant $t$ by the 
velocities $\nu_t^\text{l}$ and $\nu_t^\text{r}$ of its, respectively, left wheel and right wheel. 
The inputs of the system at time instant $t$ are the linear velocity $\nu_t$ and the angular velocity $\omega_t$, given by: 
\begin{equation}\label{eq:lin_vel}
    \nu_t=(\nu_t^\text{l}+\nu_t^\text{r})/2
\end{equation}
\begin{equation}\label{eq:rot_vel}
    \omega_t=(\nu_t^\text{r}-\nu_t^\text{l})/(2\rho)
\end{equation}
where $\rho$ represents the radius of the robot. Moreover, the admissible set of the inputs is  defined by: 

\begin{equation}\label{eq:input_cons}
    \mathbb{U}=:\Big\{[\nu,\omega]^T\in\mathbb{R}^2\Big|\frac{|\nu|}{\nu\maxu}+\frac{\rho|\omega|}{\nu\maxu} \leq 1 \Big\}
\end{equation}
where $\nu\maxu$ is the maximum absolute velocity of the wheels.%

The leader robot utilizes constant reference inputs $\nu^\text{R}$ and $\omega^\text{R}$.  
The state $\bx_t^\text{R}$ of the leader robot per time instant $t$ is characterized by the position ($p^{\text{R}}_{x,t},p^{\text{R}}_{y,t}$) of its center and by the orientation $\psi_t^\text{R}$ of the robot.%

In the continuous time domain, the state evolution for the leader robot for time instant $t$ is determined by the following non-linear system equations \eqref{eq:leader}: 
\begin{equation}
\label{eq:leader}
    \begin{bmatrix}
        \Dot{p}^{\text{R}}_{x,t}\\\Dot{p}^{\text{R}}_{y,t}\\\Dot{\psi}_t^{\text{R}}
    \end{bmatrix}=\begin{bmatrix}
        \nu^\text{R}\cos(\psi_t^\text{R})\\\nu^\text{R}\sin(\psi_t^\text{R})\\\omega^\text{R}
    \end{bmatrix}
\end{equation}
The input $\bu_t$ of the follower robot at time instant $t$ consists of the linear velocity $\nu_t$ and 
the angular velocity $\omega_t$ of this robot, while its state $\bx_t$ is determined by the position 
$(p_{x,t},p_{y,t})$ of the head\footnote{The head of the robot is referred to the front point 
of the  robot, where this point is located on the main axis of the robot that is perpendicular to the 
axis of the wheels of the robot.} of the robot 
and by its orientation $\psi_t$.  
The position is impacted by the unknown disturbance vector $\bw_t = [w_{x,t},w_{y,t}]^T$, 
such that $\sqrt{w_{x,t}^2+w_{y,t}^2}<\eta$. 
The state evolution for the follower robot is thus determined by the nonlinear system equations below:
\begin{equation}\label{eq:system2}
    \dot{\bx}_t = 
    \begin{bmatrix}
        \Dot{p}_{x,t}\\\Dot{p}_{y,t}\\\Dot{\psi_t}
    \end{bmatrix}=\begin{bmatrix}
        \nu_t\cos(\psi_t)+\rho\omega_t\sin(\psi_t)\\\nu_t\sin(\psi_t)+\rho\omega_t\cos(\psi_t)\\\omega_t
    \end{bmatrix}+\begin{bmatrix}
        w_{x,t}\\w_{y,t}\\0
    \end{bmatrix}
    =f(\bx_t,\bu_t, \bw_t)
\end{equation}

The robots were simulated for $100$ time steps and based on the e-pucks \cite{epuck}, 
with $\nu\maxu=0.13\frac{\text{m}}{\text{s}}$ and $\rho=0.0267$m. 
The aim is to maintain a constant distance $\bp^{\textrm d}=[p^{{\textrm d}}_x,p^{{\textrm d}}_y]^T$ between the follower and the leader robots, in the local coordinate frame of the leader robot (the origin of this local frame coincides with the position of the center of the leader robot and the $x$-axis is parallel to the direction of the movement of the robot). 
The leader robot moves with $\nu^{\textrm{R}}=0.015\frac{\text{m}}{\textrm{s}}$  and $\omega^\textrm{R}=0.04 \frac{\text{rad}}{\text{s}}$. 
We considered the initial state $[0,0,\frac{\pi}{3}]$ for the leader robot, the initial state $[0.4,-0.2,-\frac{\pi}{2}]^T$ 
for the follower robot, and $p^{\textrm{d}}=[-0.1,-0.1]^T$ as the desired distance between the leader and follower robots (measured in the local coordinates of the leader robot). 
The bound of the disturbances on the $x$ and $y$ positions of the follower robot is $\eta=0.004$~m. 
Note that the control framework in this problem is discrete time. 
Therefore, we used Matlab Ode45 \cite{ode45} 
in order to determine the evolved states per discrete time step, using the system of differential equations given by \eqref{eq:leader} and \eqref{eq:system2}. 

\subsubsection{Nonlinear TMPC design}

In this section, we describe the design of the nonlinear TMPC, based on the approach given in \cite{UnicycleNTMPC}, 
for controlling the motion of the follower robot.  
At every iteration, TMPC solves the following problem: 
\begin{subequations}\label{eq:NTMPC}
    \begin{equation}
     V^*(\bx_k) = \min_{\Tilde{\bz}_k,\Tilde{\bv}_k}\sum_{i=k}^{k+N-1}\left(\norm{\bz_{i|k}^\text{r}}_Q^2+\norm{\bm{v}_{i|k}^\text{r}}_R^2\right)+\norm{\bz^\text{r}_{k+N|k}}_F^2 
\end{equation}
\begin{equation*}
     \text{ s.t. for $i = k+1, \ldots, k+N$: }
\end{equation*}
\begin{equation}\label{eq:NTMPC_start}
    \bx_k\in \{\bz_{k}\}\oplus\mathbb{E}\maxu
\end{equation}
\begin{equation}\label{eq:NTMPC_evolution}
    \bz_{i+1|k}=f^\text{d}(\bz_{i|k},\bv_{i|k})
\end{equation}
\begin{equation}\label{eq:sets}
     \bv_{i|k}\in\mathbb{V},\quad \bz_{i+N|k}^\text{r}\in\mathbb{Z}_\text{f}
\end{equation}
 \begin{equation}
 \label{eq:transition_1}
        \bz^r_{i|k}[1,2]=\text{Rot}(-\bz_{i|k}[3])(\bx_i^\text{R}[1,2]-\bz_{i|k}[1:2])+\text{Rot}(\bz^\text{r}_{i|k}[3])\bp_{\textrm{d}}
    \end{equation}
    \begin{equation}
 \label{eq:transition_2}
        \bz^\text{r}_{i|k}[3]=\bx_i^\text{R}[3]-\bz_{i|k}[3]
    \end{equation}
    \begin{equation}
 \label{eq:transition_3}
        \bv^\text{r}_{i|k}=\begin{bmatrix}
            -\bv_{i|k}[1]+(\nu_\text{R}-p^{\textrm{d}}_x\omega^\text{R})\cos(\bz^\text{r}_{i|k}[3])-p^{\text{d}}_x\omega^\text{R}\sin(\bz^\text{r}_{i|k}[3])\\
            -\rho \bv_{i|k}[2]+(\nu_\text{R}-p^{\textrm{d}}_x\omega^\text{R})\sin(\bz^\text{r}_{i|k}[3])-p^{\text{d}}_y\omega^\text{R}\cos(\bz^\text{r}_{i|k}[3])
        \end{bmatrix}
    \end{equation}
\end{subequations}
where \eqref{eq:NTMPC_start} states that the initial nominal state $\bz_k$ of the follower robot, which is a decision variable, 
should be determined such that the error between the measured state $\bx_k$ and the nominal state $\bz_k$ of the follower robot belongs to set $\mathbb{E}\maxu$.  
Moreover, according to \eqref{eq:NTMPC_evolution}, the nominal states of the follower robot evolve 
according to the discretized function $f^\text{d}(\cdot)$, which is determined after time-discretization of 
\eqref{eq:system2} and by putting $\bw_k = 0$ for all time steps. 
Constraint  \eqref{eq:sets}  enforces all the calculated nominal inputs to fall within set $\mathbb{V}$, which should be constructed, such that if $\bv_{i|k} \in \mathbb{V}$, then for the actual input we have  $\bu_{i|k}\in\mathbb{U}$ for a given ancillary control law. Finally,  $\mathbb{Z}_\text{f}$ is the terminal set of the states for the follower robot.
The cost function $V(\cdot)$ is computed using $\bz^\text{r}_{i|k}$, which is the nominal state of the follower robot 
within the coordinate system of the leader robot, while the original nominal state $\bz_{i|k}$ is 
given in the global coordinates. The transition from the global coordinate system to the 
local coordinate system of the leader robot is done via \eqref{eq:transition_1}-\eqref{eq:transition_3}, where Rot$(\cdot)$ is a function that returns a $2$-dimensional rotation matrix with respect to the global coordinates, given an input angle. 
Whenever square brackets are used after a vector, if the bracket contains one scalar, the notation refers to the element 
corresponding to that scalar of the vector. In case the square brackets include `scalar~1:scalar~2', the notation refers to 
elements `scalar~1' to `scalar~2' (including these two elements) of the vector.%

The nonlinear ancillary control law in \cite{UnicycleNTMPC} is first given in the continuous time domain for a given time instant $t$, via:
\begin{equation}\label{eq:ancilary}
    \bu_t=\begin{bmatrix}
        \cos(\bx_t[3])&-\rho\sin(\bx_t[3])\\\sin(\bx_t[3])&\rho\cos(\bx_t[3])
    \end{bmatrix}^{-1}    
    \left(
    \begin{bmatrix}
        \cos(\bz_t[3])&-\rho\sin(\bz_t[3])\\\sin(\bz_t[3])&\rho\cos(\bz_t[3])
    \end{bmatrix}\bv_t-K^\text{e}\be_t[1:2]
    \right)
\end{equation}
where $K^\text{e}$ is a design variable that remains constant during the implementation of the nonlinear TMPC. 
This ancillary control law linearizes the dynamics of the position error of the controlled system. This dynamics 
is described by: 
\begin{equation}
\label{eq:pos_error_dynamics}
    \Dot{\bm{e}}_t[1:2]=K^\text{e}\bm{e}_t[1:2]+[w_x,w_y]^T
\end{equation}
After discretization, the dynamics of the position error for time step  $i\in\{k, \ldots, k+N-1\}$ is given by:
\begin{equation}\label{eq:pos_error_dynamics_disc}
    \bm{e}_{i+1}[1:2]=K^\text{e,d}\bm{e}_{i}[1:2]+\Ts[w_x,w_y]^T
\end{equation}
where $\Ts$ represents the sampling time and $K^\text{e,d}$ is calculated based on the zero-order 
hold approach (ZOH) \cite{ZOH}, in order to ensure equivalence between the continuous-time and discrete-time formulations. 
The cross section $\mathbb{E}\maxu$ of the tube for nonlinear TMPC, the set  $\mathbb{V}$ 
of admissible values for the nominal control inputs, and the terminal control law (i.e., the control law that is implemented beyond the prediction horizon) are defined by:  
\begin{equation}\label{eq:nonlinear_tube}
    \mathbb{E}\maxu=\Big\{\bm{e}\in\mathbb{R}^3\big| \max(|\bm{e}[1:2]|)<\frac{\Ts\eta}{1-K^\text{e,d}}\Big\}
\end{equation}
\begin{equation}\label{eq:nonlinear_tube_input}
    \mathbb{V}=\Big\{\bv\in\mathbb{R}^2\big|[\frac{1}{\nu\maxu},\frac{\rho}{\nu\maxu}]|\bv|<\frac{\sqrt{2}}{2}-\frac{\eta\sqrt{2}}{\nu\maxu}\Big\}
\end{equation}
\begin{equation}\label{eq:terminal_law_nonlinear}
    \bu_k^{\textrm T}=\begin{bmatrix}
        \Bar{K}\bz_{k+N}^\text{r}[1]+(\nu^\text{R}-p^{\text{d}}\omega^\text{R})\cos(\bz_{k+N}^\text{r}[3])-p^{\text{d}}_x\omega^\text{R}\sin(\bz_{k+N}^\text{r}[3])\\
        \frac{1}{\rho}(\Bar{k}\bz_{k+N}^\text{r}[2]+(\nu^\text{R}-p^{\text{d}}_x \omega^R)\sin(\bz_{k+N}^\text{r}[3])-p^{\text{d}}_y\omega^R\cos(\bz_{k+N}^\text{r}[3]))
    \end{bmatrix}
\end{equation}
where $\Bar{K}$ is a constant. 
Finally, the terminal admissible set for the position of the follower robots is given by:
\begin{equation}\label{eq:terminal_nonlinear_constriant}
    \mathbb{Z}_\text{f}=:\left\{\bz^\text{r}\in\mathbb{R}^3\big|\Bar{K}(|\bz^\text{r}[1]|+|\bz^\text{r}[2]|)
    <\frac{\nu\maxu\sqrt{2}}{2}-\eta\sqrt{2}-\sqrt{2}
    \norm{\begin{bmatrix}
        1&-p^{\text{d}}_x\\0&p^{\text{d}}_y
    \end{bmatrix}}\right\}
\end{equation}

In the numerical experiment, the following constants  were chosen: $N=10$, $\Ts=0.2$s, $Q$=diag(0.2,0.2,0), $R$=diag(0.4,0.4),  $F$=diag(0.5,0.5), $\Bar{K}=1.2$, $K^\text{e}$=2.3, $K^\text{e,d}$=0.63, which  lead to the following sets: $\mathbb{E}\maxu=\{\bm{e}\in\mathbb{R}^3|  \max(|\bm{e}[1:2]|)\leq0.0022\}$ $\mathbb{V}=0.6636\mathbb{U}$, and $\mathbb{Z}_\text{f}=\{\bz^\text{r}\in\mathbb{R}^3||z^\text{r}(1)|+|z^\text{r}(2)|\leq0.0542\}$. 

\begin{table}[]\caption{The following mathematical notation is used in section 4.2:}\resizebox{\textwidth}{!}{
\label{tab:math42}
\begin{tabular}{llll}
\hline
\multicolumn{1}{|l|}{The notation} & \multicolumn{1}{l|}{The explenation}                              & \multicolumn{1}{l|}{The notation} & \multicolumn{1}{l|}{The explenation}                                \\ \hline
diag$(\cdot)$                      & This function returns diagonal matrix with chosen elements        & set{[}i{]}                        & Projection of a set over $i^\text{th}$ dimension of the state space \\
$F$                                & Terminal cost matrix of states of the original cost function      & $T_s$                             & Sampling time                                                       \\
$f^d(\cdot)$                       & Discrtzied system function                                        & vector{[}i{]}                     & The $i^\text{th}$ element of the vector.                            \\
$\Bar{k}$                          & A constant used in the terminal control law                       & $\mathbb{V}$                      & A polytope set of feasible inputs used by controller                  \\
$k_e$                              & State-feedback constant                                           & $\mathbb{V}_\text{NL}$            & A nonlinear set which could be used to tighten input constraints    \\
$k_{ed}$                           & Discretized state-feedback constant                               & $v^r$                             & Nominal input in transformed to leader's coordinate frame           \\
m                                  & meters                                                            & $z^r$                             & Nominal state in transformed to leader's coordinate frame           \\
$\frac{\text{m}}{\text{s}}$        & meters per second                                                 & $\eta$                            & Boundary of disturbance                                             \\
$\bm{p}_d$                         & A constant distance follower should keep to the leader.           & $\lambda(\cdot)$                  & A state-dependent parameter used to tighten input constraints.      \\
${p}_{d,x}$                        & A constant distance follower should keep to the leader in x-axis. & $\nu$                             & Linear velocity of the follower                                     \\
${p}_{d,y}$                        & A constant distance follower should keep to the leader in y-axis. & $\nu_{\text{max}}$                & Maximum velocity of the robot's wheels                              \\
$p_{\text{R},x}$                   & The current leader's position component parallel to x-axis        & $\nu_{\text{l}}$                  & Velocity of the follower's left wheel                               \\
$p_{\text{R},y}$                   & The current leader's position component parallel to y-axis        & $\nu_{\text{r}}$                  & Velocity of the follower's right wheel                              \\
$p_{x}$                            & The current follower's position component parallel to x-axis      & $\nu_\text{R}$                    & Linear velocity of the leader                                       \\
$p_{y}$                            & The current follower's position component parallel to y-axis      & $\rho$                            & Radius of the robot                                                 \\
$Q$                                & Stage cost matrix of states       & $\psi$                            & The current follower's direction                                    \\
$R$                                & Stage cost matrix of inputs         & $\psi_\text{R}$                   & The current leader's direction                                      \\
ROT$(\cdot)$                       & A function that returns a rotation matrix in 2 dimensions         & $\omega$                          & Angular velocity of the follower                                    \\
s                                  & seconds                                                           & $\omega_\text{R}$                 & Angular velocity of the leader                                     
\end{tabular}}
\end{table}

\subsubsection{\dtmpc\ design}

 In this section, we present the design of \dtmpc\ for controlling the follower robot. 
 The same ancillary control law that is given by \eqref{eq:ancilary} is used for \dtmpc. 
 In case study~1, the main effort was to make the error sets $\mathbb{E}_i$ for all $i\in\{k+1,\ldots, k+N\}$ as small as possible. 
 In case study~2, however, in order to reduce the computation time, which is affected by the larger prediction horizon (which is $10$), we approximate the error sets via box sets. Although this may slightly increase the conservatism, 
 it significantly reduces the computation time.   
 Let $\mathbb{E}_i[\ell]$ denote the projection of the error set $\mathbb{E}_i$ onto the $\ell^{\text{th}}$ 
 dimension of the state space (thus for this case study $\ell=1,2,3$).  Since we are working with box sets, $\mathbb{E}_i$ is the 
 cross product over the projections of this set on all the dimensions of the state space. 
 Based on \eqref{eq:pos_error_dynamics_disc}, for $\ell = 1,2$ for the projection of the error set we can write: 
\begin{equation}\label{eq:ex}
    \mathbb{E}_{i+1}[\ell]=K^\text{e,d}\mathbb{E}_i[\ell] \oplus\mathbb{W} 
\end{equation}
where $\mathbb{W} = \{ w\in\mathbb{R}| w\leq \Ts \eta \}$.%

    The resulting position error dynamics are linear, but deriving the rotation error dynamics leads to nonlinear dynamics.
Note that computing all possible errors per time instant 
via nonlinear dynamics is generally computationally expensive. 
Moreover, we have chosen to work with polytopic error sets, while propagating polytopes through nonlinear equations does not 
necessarily result in a polytopic error set per time instant. 
Therefore, as suggested in \cite{NTMPClipshitz}, we replace the nonlinear dynamic equation 
by a simplified evolution equation for $\bm{e}_t[3]$, where the remaining terms are treated as state-dependent external disturbances \eqref{eq:error_dynamics_direction}.
    \begin{equation}\label{eq:error_dynamics_direction}
        \Dot{\bm{e}}_t[3]=-\frac{1}{\rho}\be_t[3]\bv_t[1]+\we_t \text{ where }\we_t \in \mathbb{W}_t[3](\be,\bz,\bv)
    \end{equation}
This approach has proven to be very accurate for small deviations in the third dimension of the state space (in this case 
the robot's heading). 
From now on, we will assume that the absolute values of $\bm{e}_t[3]$ remain relatively small (i.e., $|\bm{e}_t[3]|<\pi/4$). 
The exact procedure for deriving \eqref{eq:error_dynamics_direction} has been shown in Appendix \ref{app1}.

\begin{subequations} 
The final phase of the \dtmpc\ design for case study 2 is to find a method to tighten the nominal input constraints online so that the real control input satisfies the original input constraints.
To ensure this, the following input constraint \eqref{eq:v_U_sec} must be satisfied, where $\mathbb{V}^\text{NL}(\cdot)$ is a state-dependent set. Although the resulting set is a polytope if the error in direction is a known value (see a red polytope in Figure~\ref{fig:Uset}), this is not true if there is a set of possible errors in direction. In such a case, to ensure the satisfaction of the constraint, $\mathbb{V}^\text{NL}(\cdot)$ is the intersection of the output sets (see a black polytope in Figure~\ref{fig:Uset}), which generally does not lead to a polytope if the set of possible errors in direction is continuous and the exact computation of 
\eqref{eq:v_U_sec} per time instant can become computationally intractable. 

\begin{equation}\label{eq:v_U_sec}
    \{\bv_t\}\oplus\begin{bmatrix}
        -\cos(\bz_t[3])& -\sin(\bz_t[3])\\ \frac{1}{\rho}\sin(\bz_t[3])& -\frac{1}{\rho}\cos(\bz_t[3])
    \end{bmatrix} K^\text{e} (\mathbb{E}_t[1]\times\mathbb{E}_t[2])\subseteq
    \bigcap_{\be\in\mathbb{E}_t}\left(\begin{bmatrix}
        \cos (\be_t[3])&-\rho\sin(\be_t[3])\\ \frac{1}{\rho}\sin (\be_t[3])&\cos(\be_t[3])
    \end{bmatrix}\mathbb{U} \right)=\mathbb{V}^\text{NL}(\mathbb{E}_t[3])
\end{equation}

To reduce the computational complexity associated with \eqref{eq:v_U_sec} and to ensure that the resulting $\mathbb{V}^\text{NL}(\mathbb{E}_t[3])$ is always represented by a polytope, we replace the quadrangles corresponding to $\mathbb{V}^\text{NL}(\be_t[3])$ by 
(i.e., the red background quadrangle in the Figure~\ref{fig:Uset}) with the largest 
subset of $\mathbb{V}^\text{NL}(\be_t[3])$ for which the diameters are parallel to the $x$ and $y$ axes (see the blue foreground quadrangle in Figure~\ref{fig:Uset}).  
In other words, we replace $\mathbb{V}^\text{NL}(\cdot)$ with the multiplication of the original set $\mathbb{U}$  and a positive, state-dependent scalar function $\lambda(\cdot)$, i.e:
\begin{equation}
    \label{eq:new_def_V2}
     \frac{1}{\cos(\cdot)+|\sin(\cdot)|}\mathbb{U}=\lambda(\cdot)\mathbb{U}\subseteq\mathbb{V}^\text{NL}(\cdot)
\end{equation}

\begin{figure}
    \centering
    \includegraphics[width=0.95\textwidth]{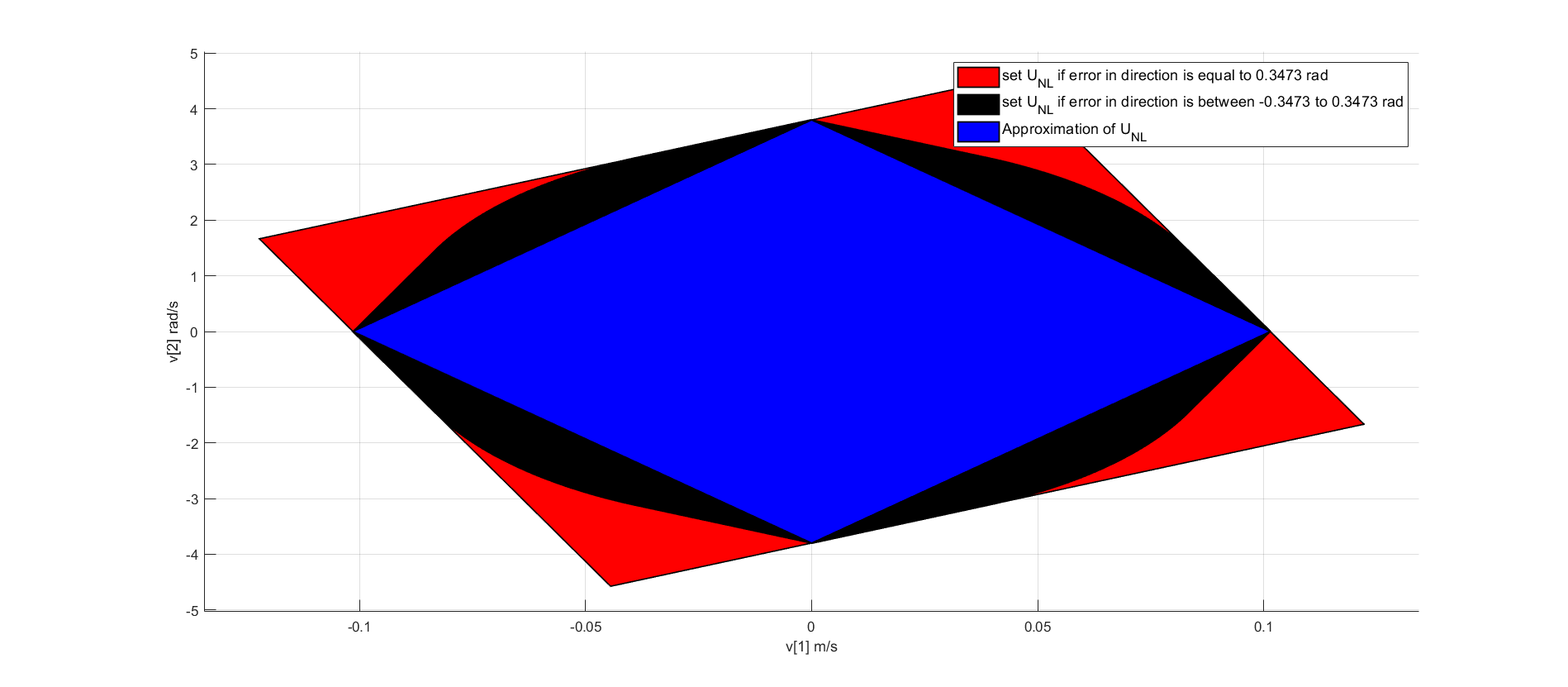}
    \caption{Illustration of an example quadrangle (the red background quadrangle) that represents $\mathbb{V}^\text{NL}(\be_t[3])$ when 
$\be_t[3] = 0.3473$. The intersection of all such quadrangles for the entire range of $\be_t[3]$, 
i.e., for $ -0.3473\leq \be_t[3] \leq 0.3473$, generates $\mathbb{V}^\text{NL}(\mathbb{E}_t[3])$ (the black middle-ground polygon).    
The simplified version of $\mathbb{V}^\text{NL}(\be_t[3])$, i.e., the largest 
subset of $\mathbb{V}^\text{NL}(\be_t[3])$ with its diameters parallel to the $x$ and $y$ axes, is shown via the blue foreground quadrangle.}
    \label{fig:Uset}
\end{figure}

Figure~\ref{fig:lambda} shows the curve corresponding to function $\lambda(\cdot)$. 
We define the admissible set $\mathbb{V}\left( \mathbb{E}_t \right)$ for the nominal control inputs as a function of the error set via:    \begin{equation}\label{eq:lambda_U2}
        \mathbb{V}\left( \mathbb{E}_t\right):=\min_{\be_t[3]\in\mathbb{E}_t[3]} \left(\lambda(\be_t[3]\right)\mathbb{U}\ominus
        \begin{bmatrix}
            -\cos(\bz_t[3])& -\sin(\bz_t[3])\\ \frac{1}{\rho}\sin(\bz_t[3])& -\frac{1}{\rho}\cos(\bz_t[3])
        \end{bmatrix} K^\text{e} (\mathbb{E}_t[1]\times\mathbb{E}_t[2])
    \end{equation}
\end{subequations}
For the design of TMPC in \cite{UnicycleNTMPC}, the admissible set of nominal control inputs was generated by scaling the set $\mathbb{U}$ by a constant value $\mathbb{U}$. with a constant value of $0.6636$. As shown in Figure~\ref{fig:lambda}, the set $\mathbb{V}\left( \mathbb{E}_t\right)$ was defined by 
over \eqref{eq:lambda_U2} is consistently larger. 
Thus, the original design is more conservative than our proposed approach. 
Note that to further reduce the conservatism, $\mathbb{V}^{\text{NL}}\left(\mathbb{E}_t\right)$ and thus $\mathbb{V}\left(\mathbb{E}_t\right)$ can be approximated by polytopes that generally have more vertices than $4$.  The exact procedure for deriving the input constraints is given in the appendix \ref{app2}.

\begin{figure}
    \centering
    \includegraphics[width=0.95\textwidth]{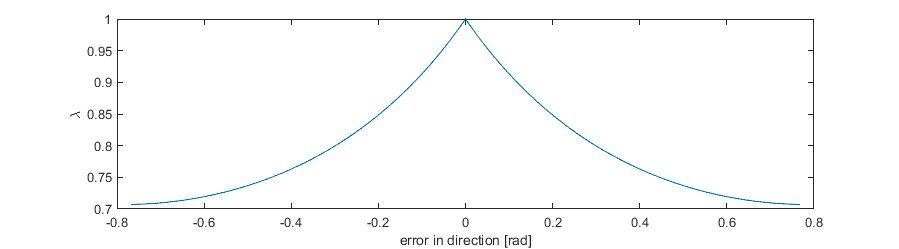}
    \caption{Illustration of the curve representing function $\lambda(\cdot)$ versus the error in the third dimension of the state space.}
    \label{fig:lambda}
\end{figure}

Finally, using the derived error dynamics given by \eqref{eq:ex} and the time-discrete version of \eqref{eq:error_dynamics_direction}, the online input constraint tightening method that is formulated by \eqref{eq:lambda_U2}, 
and using the solutions of the regular nonlinear TMPC method from \cite{UnicycleNTMPC} as a warm start, the nonlinear   
optimization problem that is solved by \dtmpc\ per time step is given by:
\begin{subequations}
\label{eq:optimization_nonlinear}
\begin{equation}
\label{eq:cost_nonlinear}
     V^*(\bx_k,\bz_k) = \min_{\Tilde{\bz}_k, \Tilde{\bv}_k , \mathbb{T}_k}\sum_{i=k}^{k+N-1}\left(\norm{\bz_{i|k}^\text{r}}_Q^2+\norm{\bm{u}_{i|k}^\text{r}}_R^2\right)+\norm{\bz^\text{r}_{k+N|k}}_F^2
\end{equation}
\begin{equation*}
     \text{ s.t. for $i = k+1, \ldots, k+N$: }
\end{equation*}
\begin{equation}
\label{eq:nominal_system_nonlinear}
    \bz_{i+1|k}=f^\text{d}(\bz_{i|k},\bv_{i|k})  
\end{equation}
 \begin{equation}
  \label{eq:tubeTk_nonlinear}
     \mathbb{T}_k = \{\{z_{k+1|k}\}\oplus\mathbb{E}_{k+1|k}, \ldots, \{z_{N|k}\}\oplus\mathbb{E}_{k+N|k} \}    
 \end{equation}
\begin{equation}
\label{eq:2nd_tube_nonlinear}
    \Bar{\mathbb{E}}_{i|k} =\text{diag} \left(K^{\text{e,d}},K^{\text{e,d}},K^{\theta,\text{d}}\right)\mathbb{E}_{i-1|k}
\end{equation}
\begin{equation}
\label{eq:final_tube_nonlinear}
    \mathbb{E}_{{i|k}} = \Bar{\mathbb{E}}_{i|k}\oplus\mathbb{W}_{i-1|k}(\mathbb{E}_{i-1|k},z_{i-1|k},v_{i-1|k})
\end{equation}
\begin{equation}
\label{eq:first_tube_nonlinear}
   \mathbb{E}_k=\{\bx_k-\bz_k\}
\end{equation}
\begin{equation}
\label{eq:policy_feasible_nonlinear}
      v_{i|k}\in\mathbb{V}(\mathbb{E}_{i|k})
\end{equation}
\begin{equation}
\label{eq:terminal_nonlinear}
  \bz^\text{r}_{N|k} \in \mathbb{Z}_\text{f}
\end{equation}
\end{subequations}
where $K^{\theta,\text{d}}$ is the dynamic coefficient that is multiplied by $\be_{\ell}[3]$, 
after \eqref{eq:error_dynamics_direction} is discretized in time, with $\ell$ the time step counter.%

For ease of comparison, we use the same horizon, sampling time, cost function, and terminal set for 
both regular nonlinear TMPC and \dtmpc. With the solution of regular nonlinear TMPC used as a warm start for \dtmpc, 
a solution is obtained for \dtmpc\ using the PS algorithm in under $30$ seconds.

\subsubsection{Numerical results}

In figure \ref{fig:path}, the path followed by both controllers is visible. It is evident that both controllers were able to reach and follow the path, despite the disturbance. Figure \ref{fig:error_theta} shows that the absolute value of error in direction is actually less than 0.6 rad (below 3.5 degrees), confirming our assumption of low angles.  

The benefit of utilizing \dtmpc\ is illustrated in figure \ref{fig:error_rise}, which depicts \dtmpc\ transitioning from a transient state to a steady state at a more rapid pace. Using \dtmpc\ robot needed over 10\% less time to decrease distance to the destination below 0.05m. In addition, once both controllers achieve the steady state, their performance is nearly equivalent, as seen in figure \ref{fig:error_steady}. This is due to the fact that \dtmpc\ is authorized to employ substantially greater nominal inputs, as shown in figure \ref{fig:input_nominal}. Even though \dtmpc\ returns significantly higher input (above 30\%), especially in the early stages, the actual input constraints are never violated (Figure \ref{fig:input_real}) and we can observe that the real input is usually 20\% higher, during the first 3 seconds of the mission, compared to the input of TMPC. This is particularly intriguing because the nominal input of \dtmpc\ violates the original constraints $\mathbb{U}$. Nonetheless, \dtmpc\ anticipates the ancillary controller's behavior and understands it can marginally violate constraints since the true input will not breach them. Those results are noteworthy as \cite{UnicycleNTMPC} compared TMPC to another robust MPC design. TMPC exhibited improved steady-state performance but required more time to achieve it. However, with \dtmpc\, we successfully eliminated the high-rise time while maintaining equivalent steady-state performance.


\begin{figure}
    \centering
    \includegraphics[width=0.95\textwidth]{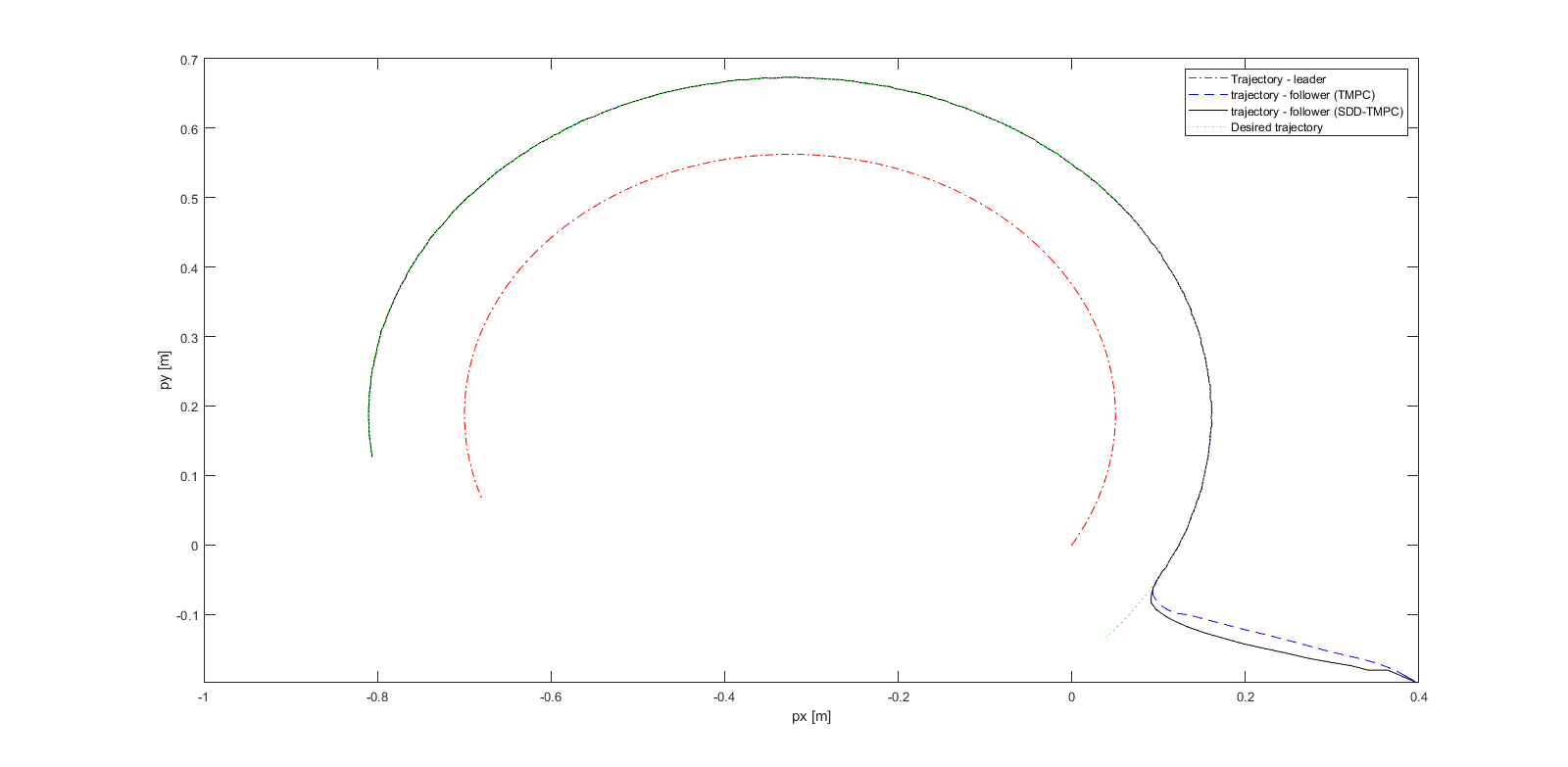}
    \caption{A desired path, paths followed by the leader, and two follower paths when the follower was using  TMPC or \dtmpc\ as controller.}
    \label{fig:path}
\end{figure}

\begin{figure}
    \centering
    \includegraphics[width=0.95\textwidth]{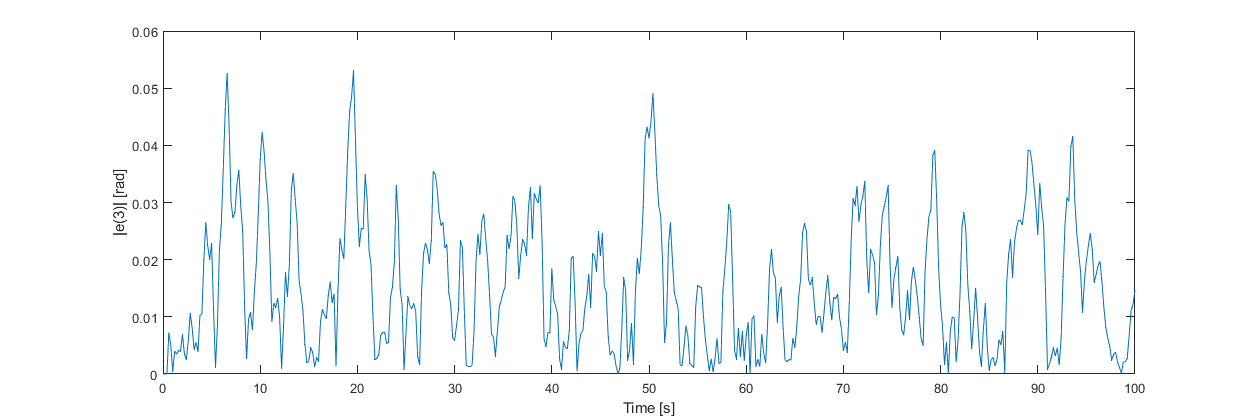}
    \caption{A graph illustrating the variance between nominal and actual direction for \dtmpc.}
    \label{fig:error_theta}
\end{figure}

\begin{figure}
    \centering
    \includegraphics[width=0.95\textwidth]{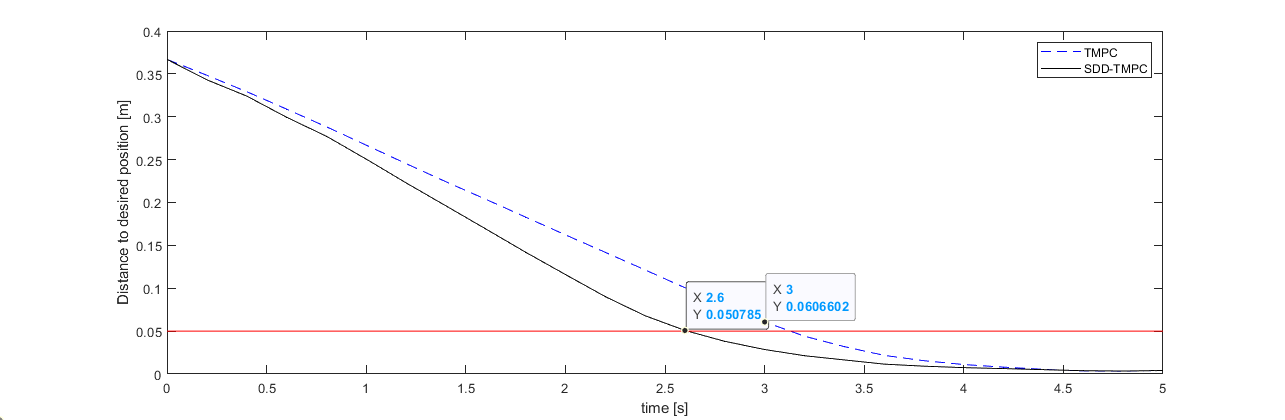}
    \caption{A distance between the desired position given time during the transient state for \dtmpc\ and TMPC.}
    \label{fig:error_rise}
\end{figure}

\begin{figure}
    \centering
    \includegraphics[width=0.95\textwidth]{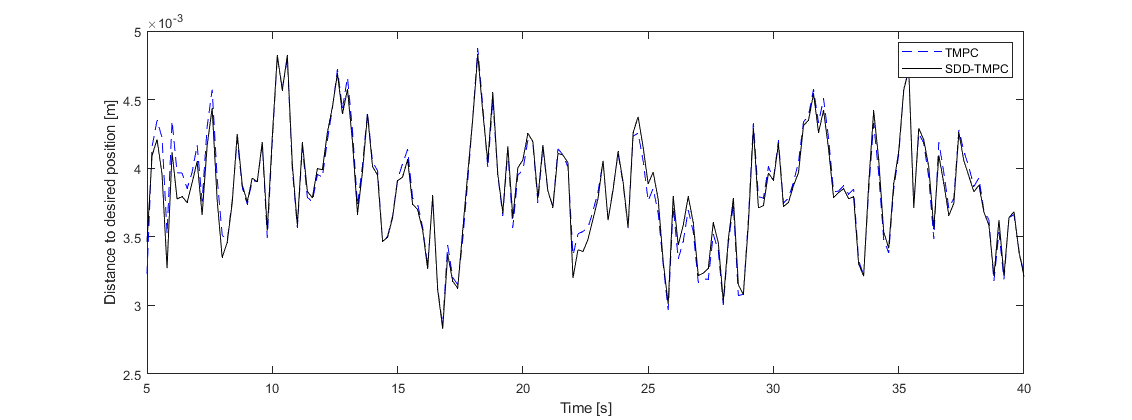}
    \caption{A distance between the desired position given time during the steady state for \dtmpc\ and TMPC.}
    \label{fig:error_steady}
\end{figure}

\begin{figure}
    \centering
    \includegraphics[width=0.95\textwidth]{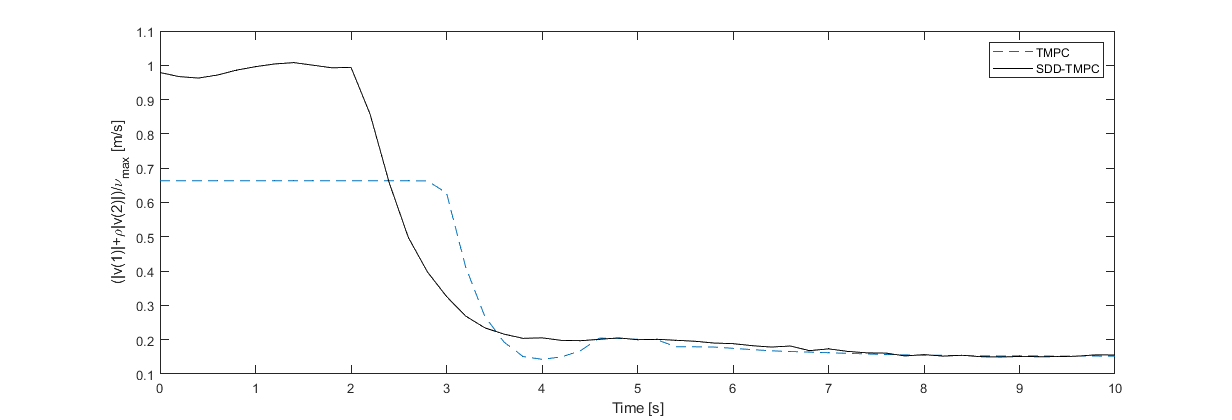}
    \caption{A nominal control input over time of \dtmpc\ and TMPC.}
    \label{fig:input_nominal}
\end{figure}

\begin{figure}
    \centering
    \includegraphics[width=0.95\textwidth]{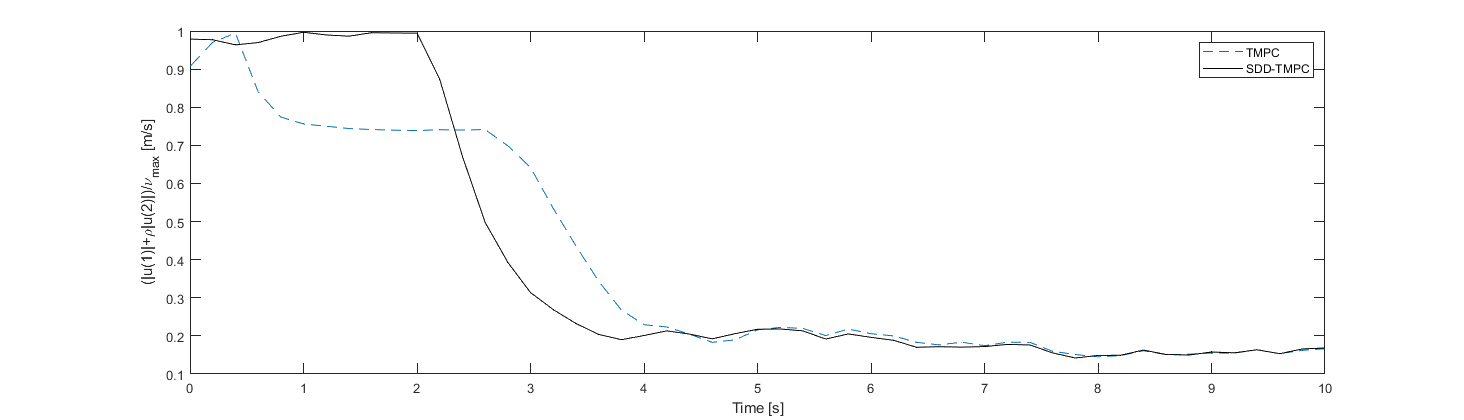}
    \caption{A control input over time of \dtmpc\ and TMPC. }
    \label{fig:input_real}
\end{figure}

}
\section{Conclusions and future work}
\label{sec:conclusion}

A dynamic version of tube model predictive control (TMPC), 
called state-dependent dynamic TMPC (\dtmpc), was proposed 
to deal with state-dependent disturbances,  
using fuzzy inference systems.
 {\dtmpc\ uses a FIS to describe how the boundaries of the disturbance change. It is possible to initialize the FIS with expert knowledge, allowing \dtmpc\ to be implemented when historical quantitative data is not available. 
To ensure robustness, such estimates should be conservative. According to our literature review, it should be possible to implement a reinforcement learning algorithm to reduce the conservativeness of this approach, but this would make it difficult to prove stability because theorem \ref{th:feasbile} would no longer hold. 

In the simulations, we used FIS tuned with offline data, which allowed us to prove the stability of \dtmpc\ under standard assumptions. In this paper, FIS was used to model a relatively simple model of disturbance to show the potential behavior of the controller. The next step could be using FIS to model more complex disturbances based on real-world data. In this paper, we used only the state data to estimate $\mathbb{W}$, but disturbances may depend on other factors. The input also affected the bounds of the disturbance in \cite{STRMPC_as_cone_program}. FIS should be extendable to encompass other factors such as inputs, time, or the state of the environment.}

A comparison among MPC, TMPC, and \dtmpc\ was performed via carefully designed simulations for the path planning of an autonomous robot. 
Compared to MPC, \dtmpc\ and TMPC show robustness to state-dependent disturbances, whereas 
\dtmpc\ compromises the optimality less than TMPC to achieve this robustness. 
Thus, \dtmpc\ reaches states that are inaccessible for TMPC and results in reduced mission time 
by realizing these states for the robot (e.g., higher velocity or positions that are closer to the obstacles).

{We also compared how to design a \dtmpc\ for a nonlinear problem. Our MPC was compared with a controller from a recent paper. Based on the simulations, our controller was able to achieve significantly better results as its input space was less restricted.

A main challenge of \dtmpc\ is the time that is required to solve the optimization problem online. This is due to the fact that the complexity of the optimization problem does not scale well with the number of states. Moreover, if polytopes are used to describe the error sets, the complexity of the optimization problem does not scale well with the size of the prediction horizon.} Future research includes improving the online computation time of \dtmpc\ by learning 
the nonlinear control policy using a neural network (see the following references for  MPC  
\cite{newTrainingMPCtoNN} and TMPC \cite{TMPCtoNN}),  
and by developing an efficient algorithm for the online tuning of fuzzy inference systems. 
An interesting topic for future work is to compare \dtmpc\ with a controller that includes 
the dynamic tube within the cost function, similarly to \cite{deepLearningTubes}. 
Finally, \dtmpc\ should be implemented and assessed via large-scale search-and-rescue scenarios.

\bibliography{sample}

\appendix

\section{Proof of directional error dynamics}\label{app1}

In order to obtain the error dynamics for the third dimension of the state space, from \eqref{eq:system2} we have:
    \begin{equation}\label{eq:error_derivative}
        \Dot{\bm{e}}_t[3]=\Dot{\bx}_t[3]-\Dot{\bz}_t[3]=\bu_t[2]-\bv_t[2]
    \end{equation}
which, together with \eqref{eq:ancilary}, and after applying the trigonometric identities \eqref{eq:trigonometric}
\begin{subequations}\label{eq:trigonometric}
    \begin{equation}
        \cos(\alpha+\beta)=\cos(\alpha)\cos(\beta)-\sin(\alpha)\sin(\beta)
    \end{equation}
    \begin{equation}
        \sin(\alpha+\beta)=\sin(\alpha)\cos(\beta)+\cos(\alpha)\sin(\beta)
    \end{equation}
\end{subequations}
as well as substituting $\bx_t[3]=\bz_t[3]+\be_t[3]$, results in: 
\begin{equation}\label{eq:error_derivative_2}
        \Dot{\bm{e}}_t[3]=-\frac{1}{\rho}\sin(\be_t[3])\bv_t[1]+(\cos(\be_t[3])-1)\bv_t[2]-\frac{K^\text{e}}{\rho}\sin(\bx_t[3])\bm{e}_t[1]-\frac{K^\text{e}}{\rho}\cos(\bx_t[3])\bm{e}_t[2]
\end{equation}

The equation \eqref{eq:error_derivative_2} can be transformed into \eqref{eq:error_derivative_simple} by treating the nonlinearity as a disturbance.

    \begin{equation}\label{eq:error_derivative_simple}
        \Dot{\bm{e}}_t[3]=-\frac{1}{\rho}\be_t[3]\bv_t[1]+\we_t
    \end{equation}
where $\we_t = \sum_{i=1}^4 \wpp{i}$, with 
$\wpp{1} = -\frac{1}{\rho}\left(\sin(\be_t[3]) - \be_t[3] \right)\bv_t[1]$ and 
$\wpp{2} = \left(\cos(\be_t[3]) - 1 \right)\bv_t[2]$ and 
$\wpp{3} = -\frac{K^\text{e}}{\rho} \sin(\bx_t[3])\be_t[1]$ and 
$\wpp{4} = -\frac{K^\text{e}}{\rho} \cos(\bx_t[3])\be_t[2]$.

\begin{subequations}  
For determining the bounds of the external disturbances, we find a lower bound and an upper bound per term $\wpp{i}$ for $i=1,2,3,4$. 
Since $\bm{e}_t[3]<\pi/4$, we can write:
    \begin{equation}\label{eq:bound_sin_pos}
    \begin{array}{l}
         w^{\textrm{e},1,\min}_t=-\frac{1}{\rho} \Big(\sin(\min(\mathbb{E}_t[3]))-\min(\mathbb{E}_t[3])\Big) \bv_t[1] \\
         w^{\textrm{e},1,\max}_t=-\frac{1}{\rho} \Big(\sin(\max(\mathbb{E}_t[3]))-\max(\mathbb{E}_t[3])\Big) \bv_t[1]
    \end{array}
    \qquad \text{ for }\bv_t[1]\geq0
    \end{equation}
    \begin{equation}\label{eq:bound_sin_neg}
       \begin{array}{l}
         w^{\textrm{e},1,\min}_t=-\frac{1}{\rho} \Big(\sin(\max(\mathbb{E}_t[3]))-\max(\mathbb{E}_t[3])\Big) \bv_t[1] \\
         w^{\textrm{e},1,\max}_t=-\frac{1}{\rho} \Big(\sin(\min(\mathbb{E}_t[3]))-\min(\mathbb{E}_t[3])\Big) \bv_t[1]
       \end{array}
    \qquad \text{ for }\bv_t[1]\leq0
    \end{equation}
We can bound $\wpp{2}$ by considering that $\cos(\be_t[3])-1$ is an even function with only one maximum and no minimum for the given range of $\be_t[3]$:
  \begin{equation}\label{eq:bound_cos}
    \begin{array}{l}
         w^{\textrm{e},2,\min}_t=\Big(\cos(\max|\mathbb{E}_t[3]|)-1 \Big)\max(\bv_t[2],0)\\
         w^{\textrm{e},2,\max}_t=\Big(\cos(\max|\mathbb{E}_t[3]|)-1\Big)\min(\bv_t[2],0)
    \end{array}
    \end{equation}
    Finally, for $\wpp{3}$  and $\wpp{4}$ we can write: 
\begin{equation}\label{eq:bound_x}
    \begin{array}{l}
         w^{\textrm{e},3,\min}_t=\displaystyle\min_{\left(\be_t[1], \be_t[3]\right)\in \mathbb{E}_t[1]\times\mathbb{E}_t[3] }\Big(-\frac{K^\text{e}}{\rho}\sin(\bz_t[3]+\be_t[3])\bm{e}_t[1]\Big) \\
          w^{\textrm{e},3,\max}_t=\displaystyle\max_{\left(\be_t[1], \be_t[3]\right)\in \mathbb{E}_t[1]\times\mathbb{E}_t[3] } \Big(-\frac{K^\text{e}}{\rho}\sin(\bz_t[3]+\be_t[3])\bm{e}_t[1]\Big)
    \end{array}
    \end{equation}
    \begin{equation}\label{eq:bound_y}
    \begin{array}{l}
         w^{\textrm{e},4,\min}_t=\displaystyle\min_{\left(\be_t[2],\be_t[3]\right) \in 
         \mathbb{E}_t[2]\times\mathbb{E}_t[3]} \Big(-\frac{K^\text{e}}{\rho}\cos(\bz_t[3]+\be_t[3])\bm{e}_t[2] \Big) \\
          w^{\textrm{e},4,\max}_t=\displaystyle\max_{\left(\be_t[2],\be_t[3]\right)\in \mathbb{E}_t[2]\times\mathbb{E}_t[3]}
          \Big(-\frac{K^\text{e}}{\rho}\cos(\bz_t[3]+\be_t[3])\bm{e}_t[2]\Big)
    \end{array}
    \end{equation}
Therefore, the admissible set of the external disturbances for \eqref{eq:error_derivative_simple} is  defined via:
    \begin{equation}\label{eq:bound_set}
        \mathbb{W}_t[3]:= \Big\{  \we_t \in\mathbb{R}\Big|\sum_{i=1}^4 w^{\textrm{e},i,\min}_t \leq \we_t \leq \sum_{i=1}^4 w^{\textrm{e},i,\min}_t \Big\}
    \end{equation}
\end{subequations}
For the discrete-time framework of the problem, \eqref{eq:bound_set} is estimated for the discrete time steps.%

\section{Proof of input constraint tightening}\label{app2}

Considering the ancillary control law that is formulated by \eqref{eq:ancilary}, imposing the hard constraint $\bu_t \in \mathbb{U}$, 
using the equality $\bx_t[3] = \bz_t [3] + \be_t[3]$, multiplying both sides of the hard constraints by $\begin{bmatrix}
            \cos(\be_t[3]) &-\rho\sin(\be_t[3])\\ \frac{1}{\rho}\sin (\be_t[3]) &\cos(\be_t[3]) 
        \end{bmatrix}$ from the left, and considering all possible combinations of $\be_t[1]$, $\be_t[2]$, and $\be_t[3]$,  
the following condition is obtained:  
    \begin{equation}\label{eq:v_U}
        \{\bv_t\}\oplus\begin{bmatrix}
            -\cos(\bz_t[3])& -\sin(\bz_t[3])\\ \frac{1}{\rho}\sin(\bz_t[3])& -\frac{1}{\rho}\cos(\bz_t[3])
        \end{bmatrix} K^\text{e} (\mathbb{E}_t[1]\times\mathbb{E}_t[2])\subseteq
        \bigcap_{\be\in\mathbb{E}_t}\left(\begin{bmatrix}
            \cos (\be_t[3])&-\rho\sin(\be_t[3])\\ \frac{1}{\rho}\sin (\be_t[3])&\cos(\be_t[3])
        \end{bmatrix}\mathbb{U} \right)=\mathbb{V}^\text{NL}(\mathbb{E}_t[3])
    \end{equation}
Note that multiplication of a matrix by a set that contains vector elements means that the mapping 
corresponding to that matrix is implemented on each vector element of the matrix. 
On the left-hand side of \eqref{eq:v_U}, the Minkowski difference of the nominal control input and 
a linear transformation of the error set is given, whereas on the right-hand side a nonlinear transformation 
of the admissible set of control inputs, i.e., $\mathbb{V}^\text{NL}(\mathbb{E}_t[3])$, should be computed. 

Picking an arbitrary element of $\mathbb{U}$, e.g., $[u[1],u[2]]^T$, after the mapping 
$\begin{bmatrix} \cos (\be_t[3])&-\rho\sin(\be_t[3])\\ \frac{1}{\rho}\sin (\be_t[3])&\cos(\be_t[3]) \end{bmatrix}$ 
is performed on this vector, we obtain a new vector $\left[\cos(\be_t[3])u[1] - \rho \sin(\be_t[3])u[2] 
, \frac{1}{\rho} \sin(\be_t[3])u[1] + \cos(\be_t[3]) u[2] \right]^T$, 
which belongs to the admissible set $\mathbb{U}$  of inputs 
in case based on \eqref{eq:input_cons}, it satisfies the following condition: 
\[
\frac{\Big|\cos(\be_t[3])u[1] - \rho \sin(\be_t[3])u[2]\Big|}{\nu\maxu } 
+ \frac{\Big|\sin(\be_t[3])u[1] + \rho \cos(\be_t[3]) u[2] \Big|} {\nu\maxu } \leq 1
\]
Therefore, in general for each $\be_t[3] \in \mathbb{E}_t[3]$ the right-hand side term of \eqref{eq:v_U}, i.e.,  
$\mathbb{V}^\text{NL}(\be_t[3])$, can be defined  by the following quadrangle:  
    \begin{equation}
       \begin{bmatrix}\label{eq:linear_transformation}
            \frac{\cos(\be_t[3]) + \sin(\be_t[3])}{\nu\maxu}&\frac{-\rho\sin(\be_t[3]) + \rho\cos(\be_t[3])}{\nu\maxu}\\
            \frac{-\cos(\be_t[3]) + \sin(\be_t[3])}{\nu\maxu}&\frac{\rho\sin(\be_t[3])+\rho\cos(\be_t[3])}{\nu\maxu}\\
            \frac{-\cos(\be_t[3])-\sin(\be_t[3])}{\nu\maxu}&\frac{\rho\sin(\be_t[3])-\rho\cos(\be_t[3])}{\nu\maxu}\\
            \frac{\cos(\be_t[3])-\sin(\be_t[3])}{\nu\maxu}& \frac{-\rho\sin(\be_t[3])-\rho\cos(\be_t[3])}{\nu\maxu}\\
        \end{bmatrix}\bu_t\leq\begin{bmatrix}
            1\\1\\1\\1
        \end{bmatrix}
    \end{equation}

To determine a subset of $\mathbb{V}^\text{NL}(\be_t[3])$ with the same shape as $\mathbb{U}$, we first put $\bu_t[1]=0$, which based on \eqref{eq:linear_transformation} results in:  
    \begin{equation}\label{eq:v10}
        \begin{bmatrix}
            \bu_t[2]\\\bu_t[2]\\-\bu_t[2]\\-\bu_t[2]
        \end{bmatrix}\leq
        \begin{bmatrix}
            \frac{\nu\maxu}{\rho(\cos(\be_t[3])-\sin(\be_t[3]))}\\\frac{\nu\maxu}{\rho(\cos(\be_t[3])+\sin(\be_t[3]))}\\
             \frac{\nu\maxu}{\rho(\cos(\be_t[3])-\sin(\be_t[3]))}\\\frac{\nu\maxu}{\rho(\cos(\be_t[3])+\sin(\be_t[3]))}
        \end{bmatrix} 	
    \end{equation}
Since $|\be_t[3]|<\pi/4$, both `$\cos(\be_t[3])-\sin(\be_t[3])$' and `$\cos(\be_t[3])+\sin(\be_t[3])$' 
result in positive values. Thus, \eqref{eq:v10} can be reformulated via:  
    \begin{equation}\label{eq:v10_simple}
        \left|\bu_t[2]\right|\leq \frac{\nu\maxu}{\rho}\lambda\left( \be_t[3] \right)
    \end{equation}
    where we define $\lambda(\cdot)$ via: 
    \begin{equation}\label{eq:lambda_def}
        \lambda(\cdot) = \min \left(\frac{1}{\cos(\cdot)-\sin(\cdot)},\frac{1}{\cos(\cdot)+\sin(\cdot)}\right)=\left(\frac{1}{\cos(\cdot)+|\sin(\cdot)|}\right)
    \end{equation}
        Similarly, if we put $\bu_t[2]=0$, from \eqref{eq:linear_transformation} we obtain:  
    \begin{equation}\label{eq:u20}
        \begin{bmatrix}
            \bu_t[1]\\\bu_t[1]\\-\bu_t[1]\\-\bu_t[1]
        \end{bmatrix}\leq
        \begin{bmatrix}
            \frac{\nu\maxu}{\cos(\be_t[3])-\sin(\be_t[3])}\\
            \frac{\nu\maxu}{\cos(\be_t[3])+\sin(\be_t[3])}\\
             \frac{\nu\maxu}{\cos(\be_t[3])-\sin(\be_t[3])}\\
             \frac{\nu\maxu}{\cos(\be_t[3])+\sin(\be_t[3])}
        \end{bmatrix} 	
    \end{equation}
    which is equivalent to the following equation:
    \begin{equation}\label{eq:u20_simple}
        \left|\bu_t[1]\right|\leq \nu\maxu \lambda\left( \be_t[3] \right)
    \end{equation}
From \eqref{eq:v10_simple} and \eqref{eq:u20_simple}, vectors $\lambda(\be_t[3]) \left[0,\frac{\nu\maxu}{\rho}\right]^T$, 
$\lambda(\be_t[3])\left[0,-\frac{\nu\maxu}{\rho}\right]^T$, $\lambda(\be_t[3])\left[\nu\maxu,0\right]^T$, and $\lambda(\be_t[3])\left[-\nu\maxu,0\right]^T$ are obtained as the corners of the polytope that represents the largest subset of $\mathbb{V}^\text{NL}(\be_t[3])$ 
(see, e.g., the blue foreground quadrangle in Figure~\ref{fig:Uset}).
All other points of this polytope are obtained by a convex combination of these vectors. 
In this case, the original set $\mathbb{U}$ is simplified to a polytope with it corners given by $[\nu\maxu,0]^T$, 
$\left[0,-\frac{\nu\maxu}{\rho}\right]^T$, $[-\nu\maxu,0]$, and $\left[0,\frac{\nu\maxu}{\rho}\right]^T$, which satisfies 
$ \lambda(\be_t[3])\mathbb{U}\subseteq\mathbb{V}^\text{NL}(\be_t[3])$.

From the expression of $\lambda(\cdot)$ given by \eqref{eq:lambda_def} and as is shown in Figure~\ref{fig:lambda}, for 
negative values of $\be_t[3]$, $\lambda(\cdot)$ shows an ascending behavior, where the maximum of $\lambda(\cdot)$ is unity, which occurs for $\be_t[3]=0$. For positive values of $\be_t[3]$, $\lambda(\cdot)$ shows a descending behavior. 
Moreover, $\lambda(\cdot)$ is an even function, which implies that its representative curve is symmetric with 
respect to the vertical axis (see Figure~\ref{fig:lambda}). 
Thus, for tightening the constraints online, from \eqref{eq:v_U}, 
we define the admissible set $\mathbb{V}\left( \mathbb{E}_t \right)$ for the nominal control inputs as a function of the error set via:
    \begin{equation}\label{eq:lambda_U}
        \mathbb{V}\left( \mathbb{E}_t\right):=\min_{\be_t[3]\in\mathbb{E}_t[3]} \left(\lambda(\be_t[3]\right)\mathbb{U}\ominus
        \begin{bmatrix}
            -\cos(\bz_t[3])& -\sin(\bz_t[3])\\ \frac{1}{\rho}\sin(\bz_t[3])& -\frac{1}{\rho}\cos(\bz_t[3])
        \end{bmatrix} K^\text{e} (\mathbb{E}_t[1]\times\mathbb{E}_t[2])
    \end{equation}

\end{document}